\documentclass[a4paper,12pt,twoside]{amsart}
\usepackage[latin1]{inputenc}
\usepackage[T1]{fontenc}
\usepackage{amsmath,amssymb}
\usepackage{amsthm}
\usepackage{mathrsfs}
\usepackage{mathtools} 
\usepackage{breqn}
\usepackage{fancyhdr}
\usepackage{graphicx}
\usepackage{esint}
\usepackage{verbatim}
\usepackage[english]{babel}
\usepackage[a4paper,top=3cm,bottom=2.5cm,left=2.4cm,right=2.4cm,bindingoffset=7mm]{geometry}
\usepackage{yfonts}
\usepackage{hyperref}
\usepackage{xcolor}
\theoremstyle{definition}
\newtheorem{definition}{Definition}[section]

\newtheorem{rmk}[definition]{Remark}
\newtheorem{remark}[definition]{Remark}

\theoremstyle{plain}
\newtheorem{lemma}[definition]{Lemma}
\newtheorem{theorem}[definition]{Theorem}
\newtheorem{proposition}[definition]{Proposition}
\newtheorem{corollary}[definition]{Corollary}

\newcommand{\res}
{\mathop{\hbox{\vrule height 7pt width .5pt depth 0pt \vrule
height .5pt width 6pt depth 0pt}}\nolimits}

\newcommand{\R}{\mathbb R}

\newcommand{\G}{\mathbb G}
\newcommand{\bbG}{\mathbb G}
\newcommand{\bbR}{\mathbb R}

\newcommand{\g}{\mathfrak g}
\renewcommand{\subset}{\subseteq}

\title[Local minimizers and gamma-convergence]{Local minimizers and Gamma-convergence for nonlocal perimeters in Carnot groups}
\author[A.\ Carbotti]{Alessandro Carbotti}
\address{Dipartimento di Matematica
	e Fisica, Universit\`a del Salento,
	Via Per Arnesano, 73100 Lecce, Italy.}
\email{alessandro.carbotti@unisalento.it}
\author[S.\ Don]{Sebastiano Don}
\address{Department of Mathematics and Statistics, P.O.\ Box 35 (MaD), FI-40014, University of Jyv\"askyl\"a, Finland.}
\email{sedon@jyu.fi}
\author[D.\ Pallara]{Diego Pallara}
\address{Dipartimento di Matematica
	e Fisica, Universit\`a del Salento, and INFN, Sezione di Lecce,
	Via Per Arnesano, 73100 Lecce, Italy.}
\email{diego.pallara@unisalento.it}
\author[A.\ Pinamonti]{Andrea Pinamonti}
\address{Dipartimento di Matematica, Universit\`a di Trento,
	Via Sommarive, 14, 38123 Povo TN, Italy.}
\email{andrea.pinamonti@unitn.it}

\keywords{Carnot Groups, calibrations, nonlocal perimeters, $\Gamma$-convergence, sets of finite perimeter, rectifiability}

\thanks{ S.D.\ has been partially supported by the Academy of Finland
	(grant 288501 ``\emph{Geometry of subRiemannian groups}'' and grant
	322898
	``\emph{Sub-Riemannian Geometry via Metric-geometry and Lie-group Theory}'')
	and by the European Research Council
	(ERC Starting Grant 713998 GeoMeG ``\emph{Geometry of Metric Groups}''). D.P.\ is member of G.N.A.M.P.A.\ of the Italian Istituto Nazionale di Alta Matematica
	(INdAM) and has been partially supported by the PRIN 2015 MIUR project 2015233N54. 
A.P.\ is member of G.N.A.M.P.A.\ of the Italian Istituto Nazionale di Alta Matematica
	(INdAM).
	The authors warmly thank Gioacchino Antonelli, Xavier Cabr\'{e} and Valerio Pagliari for interesting conversations about the problem.}

\begin{document}

\begin{abstract}
	We prove the local minimality of halfspaces in Carnot groups for a class of nonlocal functionals usually addressed as nonlocal perimeters. Moreover, in a class of Carnot groups in which the De Giorgi's rectifiability Theorem holds, we provide a lower bound for the $\Gamma$-liminf of the rescaled energy in terms of the horizontal perimeter.
\end{abstract}

\maketitle
\section{Introduction}

Given an open set $\Omega\subset\mathbb{R}^n$ and $\alpha\in (0,1)$, we define the nonlocal (or fractional) $\alpha$-perimeter of a measurable set $E\subset\mathbb{R}^n$ as the functional 
\begin{equation}
\label{eq:fracperintro}
P_\alpha(E; \Omega)\coloneqq L_{\alpha}(E^c\cap\Omega,E\cap\Omega)+L_{\alpha}(E^c\cap\Omega,E\cap\Omega^c)+L_{\alpha}(E\cap\Omega,E^c\cap\Omega^c)
\end{equation}
where 
\[                                                                               
L_{\alpha}(A,B)\coloneqq\int_{A}\int_{B} \frac 1{|x-y|^{n+\alpha}}\,dx\,dy.
\]
The notion of fractional perimeter was introduced in \cite{CafRoqSav} to study nonlocal minimal surfaces of fractional type, while a generalized notion of nonlocal perimeter defined using a positive, compactly supported radial kernel was introduced in \cite{MRT}. 
Nonlocal perimeters have been object of many studies in recent years. For example 
they are related to nonlocal (not necessarily fractional) minimal surfaces, \cite{MRT, MRT2, CinSerVal}, fractal sets, \cite{Visintin90, Visintin91, Lombardini16}, phase transition \cite{SV} and many other problems. We refer the interested reader to \cite{Dipierro,Valdinoci} for further applications and for a comparison with the standard perimeter.

Nonlocal perimeter can also be characterized in terms of the Gagliardo-Slobodeckij seminorm in the framework of fractional Sobolev spaces, see \cite{DNPV}, or in terms of Dirichlet energy associated with an extension problem for the fractional Laplacian, see \cite{CafSil}.

The limiting behavior of fractional $\alpha$-perimeters as $\alpha\to 1^{-}$ and $\alpha \to 0^+$ turns out to be very interesting. Davila showed in \cite{Davila} that for a bounded Borel set $E$ of finite perimeter the following equality holds:
\begin{equation}
\lim_{\alpha\to 1^{-}}(1-\alpha)L_{\alpha}(E^c\cap\Omega,E\cap\Omega)=cP(E;\Omega),
\end{equation}
in particular, when $\Omega=\R^n$, one has
\begin{equation}\label{Davila}
\lim_{\alpha\to 1^{-}} (1-\alpha) P_\alpha(E; \mathbb{R}^n)=cP(E)
\end{equation}
where $P(E)$ denotes the classical perimeter of $E$ in $\R^n$ and $c$ is a positive constant depending only on $n$. In the subsequent paper \cite{DFPV} the authors studied the behavior of $\alpha P_\alpha(E; \Omega)$ as $\alpha\to {0^+}$, finally in \cite{ADPM} the limiting behavior of $P_{\alpha}(E;\Omega)$ is studied in the $\Gamma$-convergence sense, see also \cite{PSV} for further extensions.


Carnot groups are connected and simply connected Lie groups whose Lie algebra $\mathfrak{g}$ is stratified, i.e., 
there are linear subspaces $\mathfrak{g}_1,...,\mathfrak{g}_s$ of $\mathfrak{g}$ such that
\begin{equation}\label{stratificazioneintro}
\mathfrak{g}=\mathfrak{g}_1\oplus...\oplus \mathfrak{g}_s,\quad [\mathfrak{g}_1,\mathfrak{g}_i]=\mathfrak{g}_{i+1},\quad
\mathfrak{g}_s\neq\{0\},\quad [\mathfrak{g}_s,\g_1]=\{0\}
\end{equation}
where $[\mathfrak{g}_1,\mathfrak{g}_i]$ denotes the subspace of ${\mathfrak{g}}$ generated by
the commutators $[X,Y]$ with $X\in \mathfrak{g}_1$ and $Y\in \mathfrak{g}_i$. \\
In the last few years Carnot groups have been 
largely studied in several respects, such as Differential Geometry \cite{CDPT}, 
subelliptic Differential Equations \cite{BonLanUgu, Folland, Folland2, Sanchez-Calle}, Complex Analysis \cite{stein} and Neuroimaging \cite{CitManSar}.
Many key results of Geometric Measure Theory in the context of metric measure spaces are based on the notion of function of bounded variation and, in particular, on sets of finite perimeter.

The local theory of perimeters in Carnot groups has then attracted a lot of interest in the literature and it is natural to address the attention to their nonlocal counterpart.
In the present paper we study nonlocal perimeters coming from a positive symmetric kernel $K\colon\G\to \mathbb{R}$ satisfying
\[
\int_\G\min\{1,d(x,0)\} K(x)\,dx<+\infty,
\]
where $d$ is the Carnot-Carath\'eodory distance on $\G$, see Definition \ref{carnotdistance}.

More precisely, given two measurable and disjoint sets $E$ and $F$ in $\G$, we consider the interaction functional
\[                                                                               
L_K(E,F)\coloneqq\int_{E}\int_{F}K(y^{-1}x)\,dx\,dy
\]
and we define the nonlocal $K$-perimeter of a measurable set $E$ inside an open set $\Omega$ as in \eqref{eq:fracperintro}, namely
\[
P_K(E;\Omega)\coloneqq L_K(E^c\cap\Omega,E\cap\Omega)+L_K(E^c\cap\Omega,E\cap\Omega^c)+L_K(E\cap\Omega,E^c\cap\Omega^c).
\]
We refer to \cite{FMPPS} and \cite{Garofalo1} for a general overview.

In the first part of the paper we provide sufficient conditions that have to be satisfied by every local minimizer of the nonlocal $K$-perimeter. Given a measurable set $E_0$ and an open set $\Omega$ in $\G$,  by a local minimizer for $P_K$ in $\Omega$ with outer datum $E_0$ we mean a measurable set $E\subseteq \G$ such that $E\setminus \Omega=E_0\setminus \Omega$ and such that
\[
P_K(E;\Omega)\leq P_K(F;\Omega), \quad \text{for every measurable $F\subseteq \G$ with $F\setminus \Omega=E_0\setminus \Omega$}.
\]
Our first main result, see Theorem \ref{th:calibrationimpliesminimizer}, reads as follows.
\begin{theorem}\label{sssa}
	Let $E_0\subseteq \G$ be a measurable set and let $\Omega\subseteq \G$ be an open set such that $P_K(E_0;\Omega)<+\infty$. Let $E\subseteq \G$ be a measurable set with $E\setminus \Omega=E_0\setminus\Omega$ and assume $\chi_E$ admits a calibration (see Definition \ref{def:calibration} below). Then $E$ is a local minimizer for $P_K$ in $\Omega$ with outer datum $E_0$.
\end{theorem}
Theorem \ref{sssa} actually holds in a slightly more general form. Indeed, it can be proved even for a natural extension of the nonlocal $K$-perimeter to all measurable functions (see  \eqref{eq:J1J2} below). Both the proof of this Theorem and the definition of calibration are inspired by the ones given in \cite{pagliari}. We also notice that, using the generalized coarea formula \eqref{eq:coareaformula}, for any local minimizer provided by Theorem \ref{sssa}, among all the minimizers, one can always find the characteristic function of a set.\\
As a consequence of Theorem \ref{sssa} we prove that a suitably defined halfspace $H$ is the unique local minimizer of $P_K$ in the unit ball $B(0,1)$ with outer datum $H\setminus B(0,1)$. 

In \cite{Cabre} it is proved that, in the Euclidean setting, every measurable set $E$ that is foliated by sub- and super- solutions adapted to $\Omega$ (see Definition \ref{def:foliation}), admits a calibration and, if some natural geometric assumption hold, the minimizer is also unique. Our Theorem \ref{th:cabre} goes exactly in this direction and follows closely \cite[Theorem 2.4]{Cabre}.

Setting $K_{\varepsilon}\coloneqq\varepsilon^{-Q}K\circ\delta_{1/\varepsilon}$,
the second part of the paper investigates the asymptotic behavior of the rescaled functionals $\frac 1{\varepsilon}P_{\varepsilon}\coloneqq \frac 1{\varepsilon}P_{K_\varepsilon}$ as $\varepsilon \to 0$ in the $\Gamma$-convergence sense. 

Berendsen and Pagliari showed in \cite{BerendsenPagliari}, that, in the Euclidean case, such $\Gamma$-limit exists in $L^1_{loc}$ and equals the Euclidean perimeter, up to a multiplicative constant.
We also mention that in \cite{ADPM} the authors proved that, in the Euclidean setting, the functional $(1-\alpha)P_\alpha$, $\Gamma$-converges in $L^1_{loc}$ to the standard perimeter $P$, up to a multiplicative dimensional constant. 
For an introduction to $\Gamma$-convergence we refer the reader to the monographs \cite{DeGiorgiDalMaso, Braides}, see also \cite{MPSC1,MPSC2} where some classical results in $\Gamma$-convergence have been extended to the case of functionals depending on vector fields. \\
The main result of the second part of the paper reads as follows (see Section \ref{preliminary} for all the missing definitions).
\begin{theorem}\label{th:gammaconvintro}
	Let $\G$ be a Carnot group satisfying property $\mathcal R$, let $\Omega$ be open and bounded and assume $K\colon \G\to [0,+\infty)$ is symmetric and radially decreasing (i.e., $K(x)=\tilde{K}(r)$, where $r=\|x\|$ and $\tilde{K}$ is 
	decreasing) and such that 
	\[
	\inf_{r>1}r^{Q+1}\widetilde K(r)>0.
	\]
	Then, there exists a positive density $\rho\colon \g_1\to (0,+\infty)$ such that, for every family $(E_\varepsilon)$ of measurable sets converging in $L^1(\Omega)$ to $E\subseteq \Omega$, one has
	\begin{equation}
	\int_\Omega \rho(\nu_E) \,dP_\G(E;\cdot)\leq \liminf_{\varepsilon\to0} \frac{1}{\varepsilon}P_\varepsilon(E_\varepsilon;\Omega).
	\end{equation}
\end{theorem}
Here $P_\G(E;\cdot)$ denotes the perimeter measure of $E$ in $\G$, $\nu_E$ denotes its horizontal normal (see Definitions \ref{def:BVfunction} and \ref{def:reducedboundary}) and $Q$ is the homogeneous dimension of $\G$. \\
Some comments are in order. The proof of Theorem \ref{th:gammaconvintro} (see proof Theorem \ref{th:gammaliminf}) follows the ideas of \cite[Section 3.3]{BerendsenPagliari}, where the authors prove the $\Gamma$-convergence of the rescaled functionals to the perimeter in the Euclidean setting. Theorem \ref{th:gammaconvintro} gives us an estimate on the $\Gamma$-liminf of the functional $\frac 1{\varepsilon}P_\varepsilon$ in terms of a density $\rho$, which is explicitly computed and does not depend on the points in the boundary of $E$, but only on the horizontal directions of its normal. For the proof of this Theorem it is essential to apply a compactness argument to families of sets with uniformly bounded $K_\varepsilon$-perimeters. The compactness criterion is given in Theorem \ref{th:compactness} and we believe it has its own independent interest.
We also notice that, in the assumptions of Theorem \ref{th:gammaconvintro}, one has to restrict both the class of Carnot groups and the class of kernels. The fact that $K$ is required to be radial and with some specific rate at infinity allows us to say that $\rho$ is indeed a strictly positive density (see Proposition \ref{prop:biacca}), while the assumption on the group $\G$ to satisfy property $\mathcal R$ allows us to consider blow-ups of sets of finite perimeter. A Carnot group $\G$ satisfies property $\mathcal R$ if every set of finite perimeter in $\G$ has rectifiable reduced boundary, i.e.\ it can be covered, up to a set of measure zero, by a countable union of intrinsically $C^1$ hypersurfaces, see Definitions \ref{def:C1G}, \ref{def:rectifiable} and \ref{def:proprietaR}. As an immediate consequence (see Remark \ref{def:proprietaR}), the validity of property $\mathcal R$ ensures that at $P_\G(E)$-almost every point of $p$ in $\G$, the family $\delta_{1/r}(p^{-1}E)$ converges in $L^1_{loc}$, up to subsequences, to a vertical halfspace with normal $\nu_E(p)$.

As we have already pointed out, the problem of understanding what is the regularity of the (reduced) boundary of a set of finite perimeter in the context of Carnot groups has only received partial solutions, so far.  Whenever property $\mathcal R$ is not assumed, only partial results about blow-up of sets of finite perimeter are available in the literature. It is proved in \cite{FSSC03} that, for any set $E\subseteq \G$ with locally finite perimeter and for $P_\G(E)$-almost every $p\in \G$, the family $\delta_{1/r}(p^{-1}E)$ converges in $L^1_{\mathrm{loc}}(\G)$ to a set of \emph{constant horizontal normal} $F$, namely a set for which there exists $\nu\in \g_1$ such that
\begin{align}\label{eq:const}
\nu\chi_F\geq 0\quad \mbox{and}\quad X\chi_F=0\quad \text{for every $X\in \g_1$ with  $X \bot \nu$},
\end{align}
in the sense of distributions.

If in addition $\G$ has step 2, or it is of type $\star$, then it is proved respectively in \cite{FSSC03} and \cite{Marchi} that, up to a left translation, every set of constant horizontal normal is really a vertical halfspace. On the other hand, still in \cite[Example 3.2]{FSSC03}, it is proved that, for general Carnot groups, condition \eqref{eq:const} does not characterize vertical halfspaces. The classification of sets  with constant horizontal normal is a challenging problem and, as far as we know, the most general result in this direction is \cite[Theorem 1.2]{ALDK}. 
We mention that in the recent paper \cite{DLDMV} the authors show that the reduced boundary of any set of locally finite perimeter in any Carnot group has a so-called \emph{cone property} that in the case of filiform groups implies rectifiability in the intrinsic Lipschitz sense. 

Finally a natural question one might ask is whether the $\Gamma$-liminf estimate given by Theorem \ref{th:gammaconvintro} can be complemented by a $\Gamma$-limsup estimate. The proofs of the $\Gamma$-limsup inequality in \cite{BerendsenPagliari} and in \cite{ADPM} rely heavily upon the convergence result by 
D\'{a}vila \cite{Davila}, whose extension to Carnot groups is, as far as we know, still an open problem, 
see \cite{MaalaPin} for some preliminary results in this directions. \\


\section{Preliminaries}\label{preliminary}

\subsection{Carnot groups}

A connected and simply connected Lie group $(\G,\cdot)$ is said to be a {\it Carnot group  of	step $s$} if its 
Lie algebra ${\mathfrak{g}}$  admits a {\it step $s$ stratification} according to \eqref{stratificazioneintro}. 
For a general introduction to Carnot groups from the point of view of the
present paper and for further examples, we refer, e.g., to \cite{BonLanUgu, Folland, ledonne, stein}.

We write $0$ for the neutral element of the group, and $xy\coloneqq x\cdot y$, for any $x,y\in \G$.\\
We fix a scalar product $\langle\cdot,\cdot\rangle$ on $\g_1$ and denote by $|\cdot|$ its induced norm. We recall that a curve $\gamma \colon [a,b]\to \bbG$ is absolutely continuous if it is absolutely continuous as a curve into $\bbR^{n}$ via composition with local charts.

\begin{definition}\label{horizontalcurve}
	An absolutely continuous curve $\gamma\colon [a,b]\to \bbG$ is said to be \emph{horizontal} if 
	\[
	\gamma'(t)\in \g_1,
	\]
	for almost every $t\in [a,b]$. The \emph{length} of such a curve is given by
	\[
	L_{\bbG}(\gamma)=\int_{a}^{b}|\gamma'(t)|dt.
	\]
\end{definition}
Chow's Theorem \cite[Theorem 19.1.3]{BonLanUgu} asserts that any two points in a Carnot group can be connected by a horizontal curve. Hence, the following definition is well-posed.
\begin{definition}\label{carnotdistance}
	For every $x,y\in \bbG$, their \emph{Carnot-Carath\'{e}odory (CC) distance} is defined by
	\[
	d(x,y)=\inf \left\{L_{\bbG}(\gamma)\colon \gamma \text{ is a horizontal curve joining } x \text{ and }y\right\}.
	\]
	We also use the notation $\|x\|=d(x,0)$ for $x\in \bbG$.
\end{definition}
We denote by 
\[
B(x,r)=\{y\in \bbG: \| y^{-1}x\| < r\}
\]
the open ball centered at $x\in \bbG$ with radius $r>0$. \\
It is well-known (see e.g.\ \cite{mitchell}) that the Hausdorff dimension of the metric space $(\G,d)$ is the so-called {\em homogeneous dimension} $Q$ of $\G$, which is given by
\[
Q\coloneqq\sum_{i=1}^s i \dim(\mathfrak{g}_i).
\]
We denote by $\mathscr H^{Q}$ the Hausdorff measure of dimension $Q$ associated with the metric $d$. The measure $\mathscr H^Q$ is a Haar measure on $\G$ (see \cite[Proposition 1.3.21]{BonLanUgu}) and we write 
\[
\int_\Omega f(x)\;dx\coloneqq\int_\Omega f(x)\;d\mathscr H^Q(x),
\]
for every measurable set $\Omega$ and every measurable function $f\colon\Omega\to\R$.

We recall here the notion of exponential map. Let $X\in \g$ and let $\gamma\colon[0,\infty)\to\G$ be the unique global solution of the Cauchy problem
\[
\begin{cases}
\gamma'(t)=X(\gamma(t)) \\
\gamma(0)=0.
\end{cases}
\]
 The exponential map
\begin{equation*}
\begin{aligned}
\exp\colon\g&\rightarrow\G \\
X&\mapsto \exp(X)\coloneqq\gamma(1)
\end{aligned}
\end{equation*}
is a diffeomorphism between the Lie algebra $\g$ and the Lie Group $\G$, and we use the notation 
$\log\colon\G\rightarrow\g$ to denote its inverse.
For any $\lambda >0$, we denote by $\delta^*_\lambda\colon\mathfrak{g}\to \mathfrak{g}$ the unique linear map such that
\[
\delta^*_\lambda{X}=\lambda^i X,\qquad \forall X\in \mathfrak{g}_i.
\]
The maps $\delta^*_{\lambda}\colon\mathfrak{g}\to\mathfrak{g}$ are Lie algebra automorphisms, i.e., $\delta^*_{\lambda}([X, Y ]) = [\delta^*_{\lambda}X, \delta^*_{\lambda}Y ]$ for all $X,Y\in \mathfrak{g}$. For every $\lambda>0$, the map $\delta^*_\lambda$ naturally induces an automorphism on the group $\delta_{\lambda}\colon\bbG\to \bbG$ by the identity $\delta_{\lambda}(x)=(\exp \circ \delta^*_{\lambda}\circ \log) (x)$. It is easy to verify that both the families $(\delta^*_{\lambda})_{\lambda>0}$ and $(\delta_\lambda)_{\lambda>0}$ are a one-parameter group of automorphisms (of Lie algebra and of groups, respectively), i.e., $\delta^*_{\lambda}\circ \delta^*_{\eta}= \delta^*_{\lambda\eta}$ and  $\delta_{\lambda}\circ \delta_{\eta}= \delta_{\lambda\eta}$ for all $\lambda, \eta>0$. The maps $\delta^*_\lambda, \delta_\lambda$ are both called \emph{dilation of factor $\lambda$}. 

Denoting by $\tau_x\colon\bbG\to\bbG$ the {\em (left) translation} by the element $x\in \G$ defined as 
\[ 
\tau_x z\coloneqq x\cdot z=xz,
\] 
we remark that the CC distance is homogeneous with respect to dilations and left invariant. More precisely, for every $\lambda>0$ and for every $x,y, z\in \G$ one has
\[
d(\delta_\lambda x,\delta_\lambda y)=\lambda d(x,y),\qquad d(\tau_x y,\tau_x z)=d(y,z).
\]
This immediately implies that $\tau_x(B(y,r))=B(\tau_x y, r)$ and $\delta_\lambda B(y,r)=B(\delta_\lambda y, \lambda r)$.
\subsection{Perimeter and rectifiability} 


We introduce the notions of perimeter, reduced boundary and rectifiability.
\begin{definition}
	\label{def:BVfunction}
	Let $\Omega$ be an open set in $\G$ and let $f\in L^1_{\mathrm{loc}}(\Omega)$. We say that $f$ has locally bounded variation in $\Omega$ ($f\in BV_{\G, \mathrm{loc}}(\Omega)$), if, for every $Y\in \mathfrak g_1$ and every open set $A\Subset \Omega$, there exists a Radon measure $Yf$ on $\Omega$ such that
	\[
	\int_{A} fY\varphi \,d\mu=-\int_A \varphi \,d(Yf),
	\]
	for every $\varphi\in C_c^1(A)$. We say that $f\in L^1(\Omega)$ has \emph{bounded variation} in $\Omega$ ($f\in BV_{\G}(\Omega)$) if $f$ has locally bounded variation in $\Omega$ and, for every basis $(X_1,\dots, X_m)$ of $\mathfrak g_1$, the total variation $|D_Xf|(\Omega)$ of the measure $D_Xf\coloneqq(X_1 f,\dots,X_m f)$ is finite. If $E$ is a measurable set in $\Omega$, we say that $E$ has locally finite (resp. finite) perimeter in $\Omega$ if $\chi_E\in BV_{\G,\mathrm{loc}}(\Omega)$ (resp. $\chi_E\in BV_{\G}(\Omega)$). In such a case, the measure $|D_X\chi_E|$ is called \emph{perimeter of} $E$ and it is denoted by $P_\G(E;\cdot)$. We also use the notation $P_\G(E;\G)\eqqcolon P_\G(E)$.
\end{definition}
The following Proposition is proved in \cite[Theorem 2.2.2]{FSSCMeyers} and \cite[Theorem 1.14]{GaroNhiCPAM}.
\begin{proposition}\label{approx}
Let $\Omega\subseteq \G$ be an open set and let $u\in BV_{\bbG}(\Omega)$. Then, there exists a sequence $(u_k)$ in  $C^{\infty}(\Omega)$ such that
\begin{itemize}
\item $u_k\to u$ in $L^1(\Omega)$;
\item $|D_Xu_k|(\Omega)\to |D_X u|(\Omega)$.
\end{itemize}
\end{proposition}
\begin{definition}
	\label{def:reducedboundary}
	 Let $ E \subseteq \G$ be a set with locally finite perimeter. We define the \emph{reduced boundary} $ \mathcal{F}E $ of $ E $ to be the set of points $ p\in \G $ such that $ P_\G(E;B(p,r))> 0 $ for all $ r > 0 $ and there exists
	\[
	\lim_{r \to 0}\dfrac{D_X\chi_E (B(p,r))}{P_\G(E;B(p,r))}=\lim_{r \to 0}\dfrac{D_X\chi_E (B(p,r))}{|D_X\chi_E|(B(p,r))} \eqqcolon \nu_E(p)\in \R^m,
	\]
	with $ |\nu_E(p)| = 1 $.
\end{definition}

\begin{definition}
	Let $\Omega\subseteq \G$ be an open set in a Carnot group $\G$. We say that a function $f\colon\Omega\to \R$ is of class $C^1_\G$ if $f$ is continuous and, for any basis $X=(X_1,\dots, X_m)$ of $\g_1$, the limit,
	\[
	X_if(x)\coloneqq\lim_{t\to 0}\frac{f(x\exp(tX_i))-f(x)}{t},
	\]
	exists and defines a continuous function for every $i=1,\ldots, m$ and any $x\in\Omega$. According to this definition we also denote by
	$\nabla_Xf\colon \Omega\to \R^m$ the vector valued function defined by
	\[
	\nabla_Xf\coloneqq (X_1f,\dots,X_mf).
	\]
\end{definition}

\begin{definition}
	\label{def:C1G}
	A set $\Sigma\subseteq \G$ is said to be a hypersurface of class $C^1_\G$ if, for every $p\in \Sigma$ there exists a neighborhood $U$ of $p$, and a function $f\colon U\to \R$ of class $C^1_\G$ such that
	\[
	\Sigma\cap U=\{q\in U: f(q)=0 \},
	\]
	and $\inf_U|\nabla_Xf|>0$, for any basis $X=(X_1,\dots,X_m)$ of $\g_1$.
\end{definition}

\begin{definition}
	\label{def:rectifiable}
	Let $E\subseteq \G$ be a measurable set. We say that $E$ is $C^1_\G$-\emph{rectifiable} (or simply \emph{rectifiable}), if there exists a family $\{\Gamma_j:j\in \mathbb N\}$ of $C^1_\G$-hypersurfaces such that
	\[
	\mathscr H^{Q-1}\left (E\setminus \bigcup_{j\in \mathbb N} \Gamma_j\right )=0,
	\]
	where $Q$ is the homogeneous dimension of $\G$ and $\mathscr H^{Q-1}$ denotes the $(Q-1)-$dimensional Hausdorff measure defined through the Carnot-Carath\'eodory distance.
\end{definition}

\begin{definition}
	For any $\nu \in \g_1\setminus \{0\}$, we define the \emph{vertical halfspace with normal $\nu$} by setting
	\[
	H_\nu\coloneqq \{x\in \G\colon\langle\pi_1\log x, \nu\rangle \geq 0\},
	\]
	where $\pi_1\colon\g\rightarrow\g_1$ is the horizontal projection on the Lie algebra.
	Notice that if $x\in \G$ is such that $\langle\pi_1\log x, \nu\rangle >0$, then $x^{-1}\in H_\nu^c$.
\end{definition}
We conclude this section with the following 
\begin{definition} Let $1\leq p\leq \infty$ and let $\Omega\subset\G$ be an open set. We set
\[
W^{1,p}_{\G}(\Omega)\coloneqq\{f\in L^p(\Omega)\colon\ X_jf\in L^p(\Omega),\ \forall j=1,\ldots, m\}.
\]
\end{definition}
\begin{definition}
The convolution of two functions in $f,g\colon\G\to \R$ is defined by
\[
(f* g)(x)\coloneqq\int_{\G}f(xy^{-1})g(y)\, dy=\int_{\G}g(y^{-1}x)f(y)\, dy,
\]
for every couple of functions for which the above integrals make sense. 
\end{definition}
\begin{remark}From this definition we see that if $L$ is any left invariant differential operator in $\G$, then $L(f * g)=f* Lg$ provided the integrals converge.  Moreover, if $\G$ is not abelian, we cannot write in general $f* Lg = Lf*g$.
\end{remark}

\section{Local minimizers and calibrations}
\noindent Throughout this section, $\G$ denotes a Carnot group and we denote by $\|x\|\coloneqq d(0,x)$, where $d$ is the CC distance introduced in Definition \ref{carnotdistance}. We however notice that the results we obtain still hold when $d(0,x)$ is replaced by any other homogeneous and symmetric norm on $\G$. We also fix a kernel  
$K\colon \G\to\R$ with the following property:
\begin{align}
\label{eq:nonnegativekernel}
&K\ge 0\qquad\qquad\quad\text{in}\,\G,\\
\label{eq:evenkernel}
&K(\xi^{-1})=K(\xi)\quad\text{for any}\,\xi\in\G,\\
\label{eq:stimaL1}
&\int_{\G}\min\{1,\|x\|\}K(x)\;dx<+\infty.
\end{align}
Define also for every measurable function $u\colon\G\to [0,+\infty]$ and every measurable set $\Omega\subseteq \G$ the functional
\begin{align}
J_K(u;\Omega)&\coloneqq\frac12\int_{\Omega}\int_{\Omega}K(y^{-1}x)|u(y)-u(x)|\;dydx+ \int_{\Omega}\int_{\Omega^c}K(y^{-1}x)|u(y)-u(x)|\;dydx \nonumber \\ \label{eq:J1J2}
&\,\eqqcolon\frac12J^1_K(u;\Omega)+J^2_K(u;\Omega).
\end{align}
We also denote by $J^i(E;\Omega)\coloneqq J^i(\chi_E;\Omega)$ for $i=1,2$. Moreover, for every measurable and disjoint sets $A,B\subseteq \G$, we define the interaction between $A$ and $B$ driven by the kernel $K$ as
\begin{equation}
\label{eq:interact}
L_K(A,B)\coloneqq\int_B\int_A K(y^{-1}x)\,dy\,dx.
\end{equation}
We set $P_K(E;\Omega)\coloneqq J_K(\chi_E;\Omega)\eqqcolon J(E;\Omega)$. Therefore,
\[
P_K(E;\Omega)=L_K(E^c\cap\Omega,E\cap\Omega)+L_K(E^c\cap\Omega,E\cap\Omega^c)+L_K(E\cap\Omega,E^c\cap\Omega^c);
\]
in particular, we have that
\[
L_K(E^c\cap\Omega,E\cap\Omega)=\frac{1}{2}J^1_K(E;\Omega),
\]
and
\[
L_K(E^c\cap\Omega,E\cap\Omega^c)+L_K(E\cap\Omega,E^c\cap\Omega^c)=J^2_K(E;\Omega).
\]
We can think of $J_K^1(\chi_E;\Omega)$ as the local part of $P_K(E;\Omega)$, in the sense that if $F$ is a measurable set such that $\mathscr H^{Q}((E\triangle F)\cap\Omega)=0$, then $J_K^1(F;\Omega)=J_K^1(E;\Omega)$.

It is worth noticing that for $\Omega=\G$ we get
\[
P_K(E;\G)=L_K(E,E^c).
\]
\begin{remark}
	For every measurable set $E\subseteq\G$ we notice that $P_K(E;\Omega)$ can also be written as
	\begin{equation}\label{eq:duetre}
	P_K(E;\Omega)=\frac 12\int_{(\G\times\G)\setminus(\Omega^c\times\Omega^c)}|\chi_E(y)-\chi_E(x)|K(y^{-1}x)\,dx\,dy.
	\end{equation}
	Indeed we can write
	\[
	\begin{aligned}
	\int_{(\G\times \G)\setminus(\Omega^c\times\Omega^c)}|&\chi_E(y)-\chi_E(x)|K(y^{-1}x)\,dx\,dy
	\\
	=& \int_{(\G\times \G)\setminus(\Omega^c\times\Omega^c)} |\chi_E(y)-\chi_E(x)|^2K(y^{-1}x)\,dx\,dy
	\\
	=&\int_{(\G\times \G)\setminus(\Omega^c\times\Omega^c)}(\chi_E(y)-\chi_E(y)\chi_E(x))K(y^{-1}x)\,dx\,dy \\
	&+\int_{(\G\times \G)\setminus(\Omega^c\times\Omega^c)}(\chi_E(x)-\chi_E(y)\chi_E(x))K(y^{-1}x)\,dx\,dy
	\\
	=&2 \int_{(\G\times \G)\setminus(\Omega^c\times\Omega^c)}\chi_E(x)\chi_{E^c}(y)K(y^{-1}x)\,dx\,dy\\
	=& 2L_K(E^c\cap \Omega, E\cap \Omega)+2L_K(E^c\cap \Omega, E\cap \Omega^c)+
	2L_K(E\cap\Omega;E^c\cap \Omega^c)\\=&2P_K(E;\Omega).
	\end{aligned}
	\]
\end{remark}

When $\G$ is the Euclidean space $\R^n$, a typical example of radial kernel satisfying \eqref{eq:nonnegativekernel}, \eqref{eq:evenkernel} and \eqref{eq:stimaL1} is given by $K(x)=|x|^{-n-\alpha}$ with $\alpha\in(0,1)$. We refer e.g.\ to \cite{Valdinoci} and references therein for an overview of the classical fractional perimeter's theory.


On the other hand, if $\G$ is a general Carnot group with homogeneous dimension $Q$, and $\|\cdot\|$ is a homogeneous norm on $\G$, then, for every $\alpha\in (0,1)$, the kernel $K\colon\G\rightarrow \R$ defined by
\[
K(\xi)\coloneqq \|\xi\|^{-Q-\alpha},
\]
satisfies conditions \eqref{eq:nonnegativekernel}, \eqref{eq:evenkernel} and \eqref{eq:stimaL1}.

A homogeneous norm that has been considered in the literature is the one associated with the sub-Riemannian heat operator, see e.g.\ to \cite{Folland, FerrariFranchi, FMPPS} for some motivations. We here briefly describe its definition. Define the map $\widetilde R_\alpha\colon \G \rightarrow [0,+\infty)$ by letting
\[
\widetilde R_\alpha(x)\coloneqq -\frac{\alpha}{2\Gamma(-\alpha/2)}\int_0^{+\infty} t^{-\tfrac \alpha 2-1}h(t,x)\,dt.
\] 
Here $h\colon[0,+\infty)\times \G\rightarrow\R$ is the fundamental solution of the sub-Riemannian heat operator
\[
\mathcal H\coloneqq \partial_t+\mathcal L,
\]
where
\[
\mathcal L\coloneqq\sum_{i=1}^m X_i^2
\]
denotes the sub-Laplacian associated with a basis $(X_1,\dots,X_m)$ of the horizontal layer $\mathfrak g_1$.
In this case one has $\widetilde R_\alpha(x^{-1})=\widetilde R_\alpha(x)$ and $\widetilde R_\alpha(\delta_\lambda x)=\lambda^{-\alpha-Q}\widetilde R_\alpha(x)$ for any $x\in \G$ and any $\lambda\geq 0$, and the quantity
\[
\|x\|_\alpha\coloneqq \left(\widetilde R_\alpha(x)\right)^{-\tfrac 1{\alpha+Q}},
\]
defines a homogeneous symmetric norm on $\G$. In particular, the kernel
\[
K_\alpha(\xi)\coloneqq \frac 1 {\|\xi\|^{Q+\alpha}_\alpha}
\]
 satisfies conditions \eqref{eq:nonnegativekernel}, \eqref{eq:evenkernel}, \eqref{eq:stimaL1} and \eqref{eq:infcappa}, and hence all the results obtained in this paper apply to the special case $K=K_\alpha$. 

We next state and prove some facts that will be useful throughout the paper. 
\begin{proposition}\label{prop:frechetkolmogorov}
	Let $\Omega\subset\G$ be an open set and let $u\in BV_\G(\Omega)$. Let $p\in \Omega$, $r>0$ such that $\overline{B(p,2r)}\subseteq \Omega$ and let $g\in B(0,r)$.
	Then
	\[
	\int_{B(p,r)}|u(x\cdot g)-u(x)|\,dx\leq d(0,g)|D_Xu|(\Omega).
	\]
	In particular, if $\Omega=\G$ and $u\in BV_\G(\G)$, one has
	\begin{equation}\label{eq:stimaglobale}
	\int_{\G}|u(x\cdot g)-u(x)|\,dx\leq d(0,g)|D_Xu|(\G),
	\end{equation}
	for every $g\in \G$.
\end{proposition}
\begin{proof}
	Fix a basis $(X_1,\dots, X_m)$ of $\g_1$. By Proposition \ref{approx} we can assume without loss of generality  that $u\in C^\infty(\Omega)$.
Let $g\in B(0,r)$ with $g\neq 0$ (if $g=0$ the thesis is trivial) and let $\delta\coloneqq d(0,g)>0$. Take a geodesic  $\gamma\colon[0,\delta]\to B(0,r)$ satisfying 
	\[
	\gamma(0)=0,\quad \gamma (\delta)=g\quad \text{ and }\quad \dot \gamma(t)=\sum_{i=1}^mh_i(t)X_i(\gamma(t))\quad \text{ for a.\ \!e.\ $t\in [0,\delta]$},
	\]
	where $(h_1,\dots,h_m)\in L^\infty([0,\delta];\R ^m)$ with $\|(h_1,\dots,h_m)\|_\infty\leq 1$. Notice that, for every $x\in \G$, the curve $\gamma_x\colon [0,\delta]\to B(x,r)$ defined by $\gamma_x(t)=x\cdot \gamma (t)$ is a geodesic joining $x$ and $x\cdot g$, and $\|\dot\gamma_x\|_\infty=\|(h_1,\dots,h_m)\|_\infty$. Therefore, for any $x\in B(p,r)$, one has 
	\[
	|u(x\cdot g)- u(x)|=\left |\int_0^{\delta} \frac{d}{dt}u(\gamma_x(t))\,dt\right | \leq \int_0^{\delta}|\nabla_X u(\gamma_x(t) )|\,dt.
	\]
	Integrating both sides on $B(p,r)$ we get
	\[
	\int_{B(p,r)}|u(x\cdot g)- u(x)|\,dx\leq\int_{B(p,r)}\int_0^\delta|\nabla_X u\left (x\cdot\gamma(t)\right )|\,dt\,dx,
	\]
	and exchanging the order of integration we conclude that
	\[
	\int_{B(p,r)}|u(x\cdot g)- u(x)|\,dx\leq \int_0^\delta\int_{B(p,r)}|\nabla_X u(x\cdot \gamma(t) )|\,d x\,dt,
	\]
	where we notice that the curve $\gamma$ depends on $g$. Since $\gamma(t)\in B(0,r)$ for all $t\in [0,\delta]$ and since $x\in B(p,r)$, then $x\cdot \gamma(t)\in B(0,2r)$ for all $t\in [0,\delta]$. Indeed, by the triangular inequality one has
	\[
	d(x\cdot \gamma(t), p)\leq d(x\cdot \gamma(t), x)+d(x,p)= d(\gamma(t),0)+d(x,p)\leq r+r=2r.
	\]
Thus, we finally get 	
	\begin{align}\nonumber
	\int_{B(p,r)}|u(x\cdot g)- u(x)|\,dx&\leq d(0,g) \int_{B(p,2r)}|\nabla_X u(x)|\,d x\\
	\label{sxa}
	&\leq d(0,g) |D_X u|(\Omega).\qedhere
	\end{align}	
	\qedhere
\end{proof}
\begin{corollary}\label{mix}
Let $u\in L^1(\G)$.  Then 
\begin{equation*}
\lim_{q\to 0}\|\tau_q u - u\|_{L^1(\G)}=0.
\end{equation*}
\end{corollary}
\begin{proof}
If $u\in C^{\infty}_c(\G)$ the conclusion follows using \eqref{eq:stimaglobale}.
Let $(u_{h})$ be a sequence in $C_c^\infty(\G)$ with $u_h\to u$ in $L^1(\G)$ and let $\varepsilon>0$. Fix $h$ be big enough so that $\|u-u_h\|_{L^1(\G)}\leq \frac{\varepsilon}{4}$. Then
\begin{align*}
\|\tau_q u- u\|_{L^1(\G)}&\leq \|\tau_q u - \tau_q u_h\|_{L^1(\G)}+\|\tau_q u_h - u_h\|_{L^1(\G)}+\| u_h - u\|_{L^1(\G)}\\
&=2\| u - u_h\|_{L^1(\G)}+\|\tau_q u_h - u_h\|_{L^1(\G)}\\
&\leq\frac{\varepsilon}{2}+\|\tau_q u_h - u_h\|_{L^1(\G)}
\end{align*}
and the conclusion follows taking $d(0,q)$ small enough to have
$\|\tau_q u_h - u_h\|_{L^1(\G)}\leq \frac{\varepsilon}{2}$.
\end{proof}

We now give a sufficient condition on $E$ and $\Omega$ in order to have $P_K(E;\Omega)<+\infty$. The proof is inspired by the one present in \cite{pagliari}.
\begin{proposition}\label{prop:perimetroKfinito}
	Let $E,F\subseteq \G$ be two measurable sets with $\mathscr H^{Q}(E\cap F)=0$. Then one has
	\[
	L_K(E,F)\leq V(E,F)\int_\G \min\{1,d(0,\xi)\}K(\xi)\,d\xi,
	\] 
	where 
	\[
	V(E,F)\coloneqq\min\left\{\max\left\{\frac {P_\G(E)}2,\mathscr H^Q(E)\right\},\max\left\{\frac {P_\G(F)}2,\mathscr H^Q(F)\right\}\right\}.
	\]
\end{proposition}
\begin{proof}
	Without loss of generality we can assume
	\[
	V(E,F)=\max\displaystyle\left\{\frac {P_\G(E)}2,\mathscr H^Q(E)\right\}<+\infty.
	\]
	Up to modifying $E$ on a set of measure zero we can also assume that $F\subseteq E^c$. Therefore we have
	\begin{equation}\label{eq:perimetrovolume}
	\begin{aligned}
	L_K(E,F)\leq L_K(E,E^c)&=\frac 12\int_\G\int_\G K(\xi^{-1}\eta)|\chi_E(\xi)-\chi_E(\eta)|\,d\eta\,d\xi\\
	&= \frac 12\int_\G\int_\G K(\xi)|\chi_E(\eta\xi)-\chi_E(\eta)|\,d\eta\,d\xi\\
	&=\frac 12\int_{B(0,1)} K(\xi)\int_\G |\chi_E(\eta\xi)-\chi_E(\eta)|d\eta\, d\xi\\
	&\hphantom{=}+\frac 12\int_{\G \setminus B(0,1)} K(\xi)\int_\G |\chi_E(\eta\xi)-\chi_E(\eta)|d\eta\, d\xi.
	\end{aligned}
	\end{equation}
	Since $E$ has finite perimeter in $\G$, by Proposition \ref{prop:frechetkolmogorov} for every $\xi\in\G$ we can write
	\[
	\int_\G |\chi_E(\eta\xi)-\chi_E(\eta)|d\eta\leq d(0,\xi)P_\G(E).
	\]
	On the other hand, since $\mathscr H^Q(E)<+\infty$, we can also write
	\[
	\int_\G |\chi_E(\eta\xi)-\chi_E(\eta)|d\eta\leq 2 \mathscr H^Q(E).
	\]
	Using this two facts in the last part of \eqref{eq:perimetrovolume} gives us
	\[
	L_K(E,F)\leq \frac {P_\G(E)}2\int_{B(0,1)} d(0,\xi)K(\xi)\,d\xi+\mathscr H^Q(E)\int_{\G\setminus B(0,1)}K(\xi)\, d\xi,
	\]
	and therefore
	\[
	L_K(E,F)\leq \max\left\{ \frac{P_\G(E)}2, \mathscr H^Q(E)\right\}\int_\G \min\{1,d(0,\xi)\}K(\xi)\,d\xi. \qedhere
	\]
\end{proof}

Now, we adapt the notion of nonlocal calibration given in \cite{pagliari} in the Euclidean setting. We refer to \cite{Cabre} to point out the link between such a notion and the notion of (local) calibration of a set.

\begin{definition}\label{def:calibration}
	Let $u\colon\G\to [0,1]$ and $\zeta\colon\G\times\G\to[-1,1]$ be measurable functions. We say that $\zeta$ is a \emph{calibration} for $u$ if the following two facts hold.
	\begin{itemize}
		\item[(i)] The map $F_{\varepsilon}(p)=\int_{\mathbb{G}\setminus B(p,\varepsilon)}K(y^{-1}p)(\zeta(y,p)-\zeta(p,y))\;dy$ is such that
		\begin{equation}\label{eq:limiteazero}
		\lim_{\varepsilon\to0}\|F_\varepsilon\|_{L^1(\G)}=0.
	    \end{equation}
	  
	    \item[(ii)] for almost every $(p,q)\in \G\times\G$ such that $u(p)\neq u(q)$ one has
	    \begin{equation}\label{eq:identitacalibra}
	    \zeta(p,q)(u(q)-u(p))=|u(q)-u(p)|.
	    \end{equation}
	\end{itemize}
\end{definition}

\begin{remark}\label{rem:antisymmetric}
	If $\zeta\colon\G\times\G\to[-1,1]$ is a calibration for $u\colon\G\to[0,1]$, then also the antisymmetric function $\widehat\zeta(p,q)\coloneqq \frac12(\zeta(p,q)-\zeta(q,p))$ is a calibration for $u$.
\end{remark}
The proof of the following Theorem follows closely the one given in  \cite[Theorem 2.3]{pagliari}.
\begin{theorem}\label{th:calibrationimpliesminimizer} 
	Let $\Omega\subset\G$ be an open set and let $E_0\subseteq \G$ be a measurable set such that $P_K(E_0;\Omega)<+\infty$ and define 
	\begin{equation}
	\label{eq:setofcompetitors}
	\mathcal F\coloneqq \{v\colon\bbG\to [0,1]\ \text{measurable}\ |\  v=\chi_{E_0} \text{ on } \Omega^c\}.
	\end{equation}
	Let $u\in \mathcal F$ and let $\zeta\colon \G\times\G\to [-1,1]$ be a calibration for $u$. Then
	\[
	J_K(u;\Omega)\leq J_K(v;\Omega),
	\]
	 for every $v\in \mathcal F$. Moreover, if $\widetilde u \in \mathcal F$ is such that $J_K(\widetilde u;\Omega)\leq J_K(u;\Omega)$, then $\zeta$ is a calibration for $\widetilde u$.
\end{theorem}
\begin{proof}
	We can assume without loss of generality that $J_K(v;\Omega)<+\infty$ for every $v\in \mathcal F$. Since $|v(y)-v(x)|\geq \zeta(x,y)(v(y)-v(x))$ we can write for any $v\in \mathcal F$
	\[
	J_K(v;\Omega)\geq a(v)-b_1(v)+b_0,
	\]
	where $a,b_1$ and $b_0$ are respectively defined by
	\[ 
	\begin{aligned}
	a(v)&\coloneqq \frac12\int_\Omega\int_\Omega K(y^{-1}x)\zeta(x,y)(v(y)-v(x))\,dy\,dx,\\
	b_1(v)&\coloneqq\int_\Omega\int_{\Omega^c}K(y^{-1}x)\zeta(x,y)v(x)\,dy\,dx,	\\
	b_0&\coloneqq \int_\Omega\int_{\Omega^c}K(y^{-1}x)\zeta(x,y)\chi_{E_0}(y)\,dy\,dx.
	\end{aligned}
	\]
	By \eqref{eq:identitacalibra}, we notice that $J_K(u;\Omega)= a(u)-b_1(u)+b_0$. It is then enough to prove that, for every $v\in \mathcal F$, one has $a(v)=b_1(v)$. By Remark \ref{rem:antisymmetric}, we can assume that $\zeta$ is antisymmetric. Combining this with the fact that $K(\xi^{-1})=K(\xi)$, we easily get
	\begin{equation}\label{eq:a}
	a(v)=-\int_\Omega\int_\Omega K(y^{-1}x)\zeta(x,y)v(x)\;dydx.
	\end{equation}
	By \eqref{eq:limiteazero}, for almost every $x\in \Omega$, we have
	\[
	\begin{aligned}
	\lim_{r\to0}\int_{B(x,r)^c}&K(y^{-1}x)\zeta(x,y)\;dy\\ &=\lim_{r\to0}\int_{B(x,r)^c\cap \Omega}K(y^{-1}x)\zeta(x,y)\;dy+\int_{\Omega^c}K(y^{-1}x)\zeta(x,y)\;dydx=0.
	\end{aligned}
	\]
	Implementing this identity in \eqref{eq:a} and using the dominated convergence Theorem, we get 
	\[
	\begin{aligned}
	a(v)&=-\int_\Omega\int_\Omega K(y^{-1}x)\zeta(x,y)v(x)\,dy\,dx\\ &=-\lim_{r\to0}\int_\Omega\int_{B(x,r)^c\cap\Omega} K(y^{-1}x)\zeta(x,y)v(x)\,dy\,dx \\
	& =\int_\Omega\int_{\Omega^c} K(y^{-1}x)\zeta(x,y)v(x)\,dy\,dx=b_1(v).
	\end{aligned}
	\]
	We are left to prove that, if $\widetilde u\in \mathcal F$ is such that $J_K(\widetilde u;\Omega)\leq J_K(u;\Omega)$, then $\zeta$ is a calibration for $\widetilde u$. Since $u=\widetilde u$ on $\Omega^c$ we get 
	\begin{equation}\label{eq:identitautilde}
	\zeta(x,y)(\widetilde u(y)-\widetilde u(x))=|\widetilde u(y)-\widetilde u(x)|,
	\end{equation}
	for almost every $(x,y)\in \Omega^c\times\Omega^c$ satisfying $u(x)\neq u(y)$. Since $J_K(\widetilde u;\Omega)=b_0$, we also have that $J_K(\widetilde u;\Omega)=a(\widetilde u)-b_1(\widetilde u)+b_0$. This implies that
	\[
	\begin{aligned}
	\frac12&\int_\Omega\int_\Omega K(y^{-1}x)\left(|\widetilde u(y)-\widetilde u(x)|-\zeta(x,y)(\widetilde u(y)-\widetilde u(x))\right)\,dy\,dx \\+& \int_\Omega\int_{\Omega^c} K(y^{-1}x)\left(|\widetilde u(y)-\widetilde u(x)|-\zeta(x,y)(\widetilde u(y)-\widetilde u(x))\right)\,dy\,dx=0.
	\end{aligned}	
	\]
	Since both integrands are positive, we get that \eqref{eq:identitautilde} holds true for almost every $(x,y) \in \Omega \times \G$ with $\widetilde u(x)\neq \widetilde u(y)$. To get \eqref{eq:identitautilde} for almost every $(x,y)\in \Omega^c\times \Omega$ it is enough to use the antisymmetry of $\zeta$.
\end{proof}

We now notice that the functional $J_K(\cdot;\Omega)$ enjoys a coarea formula. Concerning the Euclidean case, we refer the reader to \cite[Theorem 2.93]{AFP} for the classical formula relating total variation and Euclidean perimeter, and to \cite{Visintin91}, where the author finds a class of functionals defined on $L^1(\Omega)$ for which a generalized coarea formula holds.

\begin{proposition}\label{prop:coarea}
	Let $\Omega\subseteq\G$ be an open set and let $u\colon\Omega\rightarrow [0,1]$ be a measurable function. Setting $E_t\coloneqq\{g\in \G: u(g)>t\}$ for any $t\in[0,1]$, it holds that
	\begin{equation}
	\label{eq:coareaformula}
	J_K(u;\Omega)=\int_0^1 P_K(E_t;\Omega)\,dt.
	\end{equation}
\end{proposition}
\begin{proof}
	The proof is analogous to the Euclidean case, see \cite[Lemma 6.2.]{CinSerVal}. Fix $x,y\in \Omega$ with $x\neq y$ and assume without loss of generality that $u(x)>u(y)$. Then $|\chi_{E_t}(x)-\chi_{E_t}(y)|=1$ for any $t\in[u(y),u(x)]$ and $|\chi_{E_t}(x)-\chi_{E_t}(y)|=0$ for every $t\in [0,1]\setminus [u(y),u(x)]$. Therefore, for any $x,y\in \Omega$, it holds that
	\[
	|u(x)-u(y)|=\left|\int_{u(y)}^{u(x)}|\chi_{E_t}(x)-\chi_{E_t}(y)|\,dt\right|=\int_0^1 |\chi_{E_t}(x)-\chi_{E_t}(y)|\,dt.
	\]
	Now, using Tonelli's Theorem, we have that
	\begin{equation*}
	\begin{aligned}
	J_K(u;\Omega)&=\int_0^1\left[\frac{1}{2}\int_\Omega\int_\Omega|\chi_{E_t}(x)-\chi_{E_t}(y)|K(y^{-1}x)\,dx\,dy\right]dt \\ &\hphantom{=}+\int_0^1\left[\int_\Omega\int_{\Omega^c}|\chi_{E_t}(x)-\chi_{E_t}(y)|K(y^{-1}x)\,dx\,dy\right]dt=\int_0^1 P_K(E_t;\Omega)\,dt. \qedhere
	\end{aligned}
	\end{equation*}
\end{proof}

As an immediate consequence of Corollary \ref{cor:coarea} below we deduce that, if the infimum of for $J_K(\cdot;\Omega)$ 
with outer datum $E_0$ is achieved, there is always a minimizer which is the characteristic function of a measurable set.
\begin{corollary}\label{cor:coarea}
	Let $\Omega\subset\G$ be an open set and let $v\colon\G\to [0,1]$ be a measurable function. Then, there exists a measurable set $F$ such that
	\[
	J_K(F;\Omega)\le J_K(v;\Omega).
	\]
\end{corollary}
\begin{proof}
	Denote by $E_t\coloneqq \{g\in \G:v(g)>t\}$. By the coarea formula \eqref{eq:coareaformula}, there exists $t^\star\in[0,1]$ such that $J_K(E_{t^\star};\Omega)\leq J_K(v;\Omega)$, otherwise the equality in \eqref{eq:coareaformula} would be contradicted. In particular, setting $F\coloneqq E_{t^\star}$, the proof is complete.
\end{proof}
In Proposition \ref{prop:findcalibration} below we show that halfspaces in Carnot group admit a calibration. In Theorem \ref{th:localminimalitycalibration}, we show that halfspaces are unique local minimizers for $J_K(\cdot;\Omega)$ when subject to their own outer datum and whenever $\Omega$ is a ball centered at the origin. 
\begin{proposition}\label{prop:findcalibration}
	For any $\nu\in\g_1\setminus\{0\}$, the map $\zeta_\nu\colon\G\times\G\to[0,1]$ defined by 
	\[
	\zeta_\nu(x,y)\coloneqq\mathrm{sign}\left(\langle \pi_1 \log(x^{-1}y),\nu \rangle \right),
	\]
	is a calibration for $\chi_{H_\nu}$.
\end{proposition}
\begin{proof}
	Denote for shortness $H=H_\nu$ and $\zeta=\zeta_\nu$. Let us first prove property $(ii)$ of Definition \ref{def:calibration}, namely that for almost every $(x,y)\in \G\times\G$ with $\chi_H(x)\neq \chi_H(y)$ one has
	\[
	\zeta(x,y)(\chi_H(y)-\chi_H(x))=|\chi_H(y)-\chi_H(x)|.
	\]
	It is not restrictive to assume that $x\in H$ and $y\in H^c$. Then
	\[
	\langle\pi_1\log (x^{-1}y), \nu\rangle=-\langle\pi_1\log x, \nu\rangle +\langle\pi_1\log y, \nu\rangle<0.
	\]
	Concerning property $(i)$ of Definition \ref{def:calibration} we observe that for every $r>0$ and every $x\in \G$ one has
	\[
	\begin{aligned}
	\int_{\G\setminus B(x,r)}&K(y^{-1}x)\left(\mathrm{sign}(\langle\pi_1\log( x^{-1}y), \nu\rangle )-\mathrm{sign}(\langle\pi_1\log (y^{-1}x), \nu\rangle)\right)dy\\
	&=2\int_{\G\setminus B(x,r)\cap xH} K(y^{-1}x)\;dy-2\int_{\G\setminus B(x,r)\cap xH^c} K(y^{-1}x)\;dy\\
	&=2\int_{\G\setminus B(0,r)\cap H} K(z)\;dz-2\int_{\G\setminus B(0,r)\cap H^c} K(z)\;dz=0.
	\end{aligned}
	\]
	The last identity comes from the fact that $\mathscr H^Q(\{x\in \G:\langle \pi_1\log x, \nu\rangle=0\})=0$, $K(x^{-1})=K(x)$ and the inversion $\xi\mapsto \xi^{-1}$ preserves the volume and maps $H$ onto $H^c$ (up to sets of measure zero).
\end{proof}

\begin{theorem}\label{th:localminimalitycalibration}
	Let $H$ be a vertical halfspace and denote by $B\coloneqq B(0,1)$. Then
$P_K(H;B)<+\infty$ and 
	\[
	P_K(H;B)\leq J_K(v;B),
	\]
	for every measurable $v\colon\G\to [0,1]$ such that $v=\chi_H$ almost everywhere on $B^c$. Moreover, if $u\colon\G\to [0,1]$ is such that $u=\chi_H$ almost everywhere on $B^c$ and $J_K(u;B)\leq J_K(\chi_H;B)$, then $u=\chi_H$ almost everywhere on $\G$.
\end{theorem}
\begin{proof}
	By definition of $P_K$ we can write
	\[
	P_K(H;B)=L_K(H\cap B, H^c\cap B^c)+ L_K(H^c\cap B, H\cap B^c)+L_K(H^c\cap B,H\cap B)<+\infty,
	\]
	since each term on the right-hand side is finite because of Proposition \ref{prop:perimetroKfinito}.
	
	By Proposition \ref{prop:findcalibration} and Theorem \ref{th:calibrationimpliesminimizer} we only have to show that minimizers are unique (up to sets of measure zero). Let $\nu \in \g_1\setminus\{0\}$ be such that $H=H_\nu$ and let $u\colon\G\to [0,1]$ be such that $u=\chi_H$ almost everywhere on $B^c$ and $J_K(u;B)\leq J_K(\chi_H;B)$. Consider the map $\zeta(x,y)=\mathrm{sign}(\langle \pi_1\log (x^{-1}y),\nu\rangle)$ which is a calibration of $\chi_H$. By Theorem \ref{th:calibrationimpliesminimizer}, $\zeta$ is also a calibration for $u$ and therefore
	\[
    \mathrm{sign}(\langle \pi_1\log (x^{-1}y),\nu\rangle)(u(y)-u(x))=|u(y)-u(x)|, \quad \text{for a.e.\ $(x,y)\in \G\times\G$.}
	\]
	 As a consequence, the implication
	 \[
	 \langle \pi_1\log (x^{-1}y),\nu\rangle>0 \Rightarrow u(y)\geq u(x)
	 \]
	holds for almost every $(x,y)\in\G\times\G$.
	For every $t\in (0,1)$, define the set $E_t\coloneqq \{\xi\in\G: u(\xi)>t\}$. For almost every $(x,y)\in E_t\times E_t^c$ one has $u(x)>u(y)$ and therefore $\langle \pi_1\log x,\nu\rangle\geq\langle \pi_1\log y,\nu\rangle$. By Dedekind's Theorem, and up to sets of measure zero, for every $t\in (0,1)$, there exists $\lambda_t\in \R$ such that $E_t\subseteq \{\xi\in\G:\langle \pi_1\log \xi,\nu\rangle\geq \lambda_t\}$ and $E_t^c\subseteq\{\xi\in\G:\langle \pi_1\log \xi,\nu\rangle\leq \lambda_t\}$. \\
	This implies that for all $t\in (0,1)$ one has
	\[
	\mathscr H^Q(E_t\triangle\{\xi\in\G:\langle \pi_1\log \xi,\nu\rangle\geq \lambda_t\})=0.
	\]
	Combining this with the fact that $u=\chi_H$ almost everywhere on $B^c$, we get that $\lambda_t=0$ for every $t\in (0,1)$, and therefore
	\begin{equation}\label{eq:misurazero}
	\mathscr H^Q(E_t\triangle H)=0, \qquad\forall t\in (0,1).
	\end{equation}
	 Consider now a sequence $(t_j)$ in $(0,1)$ that converges to $0$ as $j\to+\infty$. Since $u$ has values in $[0,1]$, we get 
	\[
	\{\xi\in\G:u(\xi)\leq 0\}=\{\xi\in\G:u(\xi)= 0\}=\bigcap_{j\in\mathbb N}E_{t_j}^c,
	\]
	and similarly
	\[
	\{\xi\in\G:u(\xi)=1\}= \bigcap_{j\in \mathbb N}E_{1-t_j}.
	\]
	Combining this fact with \eqref{eq:misurazero}, we complete the proof by observing that the identities $\mathscr H^Q(\{\xi\in \G:u(\xi)=0\}\triangle H^c)=0$ and $\mathscr H^Q(\{\xi\in \G:u(\xi)=1\}\triangle H)=0$ hold. 
\end{proof}	
Now, we show another useful approach for the analysis of the minimizers of the functional in \eqref{eq:J1J2}. To do this, following \cite{Cabre} we introduce the notion of nonlocal mean curvature and of calibrating functional; these tools allow us to prove Theorem \ref{th:cabre}. We more precisely clarify the relation between Theorem \ref{th:calibrationimpliesminimizer} and Theorem \ref{th:cabre} in Remark \ref{rem:similarities}.

\begin{definition}
	Let $E$ be a set of finite perimeter in $\G$. The \emph{nonlocal mean curvature} is defined as
	\begin{equation}
	H_K[E](x)\coloneqq\lim_{\varepsilon\to 0}\int_{\G\setminus B_\varepsilon(x)}(\chi_{E^c}(y)-\chi_{E}(y))K(y^{-1}x)\,dy.
	\end{equation}
	More generally, for every measurable map $\phi\colon\G\to \R$ we set
	\[
	H_K(\phi)(x)\coloneqq\lim_{\varepsilon\to 0}\int_{\G\setminus B_\varepsilon(x)}\mathrm{sign}(\phi(x)-\phi(y))K(y^{-1}x)\,dy.
	\]
	Notice that $H_K[\phi](x)=H_K(\{\phi>\phi(x)\})(x)$.
\end{definition}
\begin{definition}
	\label{def:foliation}
	Let $\Omega\subset\G$ be a bounded open set and let $E$ be a measurable set. We say that $\Omega$ is foliated by sub- and super- solutions adapted to $E$ whenever there exists a measurable function $\phi_E\colon\G\rightarrow \R$ such that
	\begin{enumerate}
		\item[(i)] $E=\{\phi_E(x)>0\}$ up to $\mathscr H^Q$-negligible sets;
		\item[(ii)] The sequence of functions $F_h(x)\coloneqq\int_{\G\setminus B(x, 1/h)}\mathrm{sign}(\phi_E(x)-\phi_E(y))K(y^{-1}x)\,dy$ converges in $L^1(\Omega)$ to $H_K(E)$ as $h\to \infty$.
		\item[(iii)] $H_K[\phi_E](x)\le 0$ for a.e.\ $x\in\Omega\cap E$ and $H_K[\phi_E](x)\ge 0$ for a.e.\ $x\in\Omega\setminus E$.
	\end{enumerate} 
	Notice that the integrals defined in (ii) are finite thanks to the assumptions on the kernel.
\end{definition}  

\begin{definition}
	Let $\Omega\subseteq \G$ be a bounded open set and let $E\subseteq \G$ be measurable. Assume that $\Omega$ is foliated by sub- and super- solutions adapted to $E$ and let $\phi_E$ be the measurable function provided by Definition \ref{def:foliation}. Then, for every measurable set $F$ such that $F\setminus\Omega=E\setminus\Omega$, we define the \emph{calibrating functional} as
	\begin{equation}
	\label{eq:calibrationcabre}
	\mathcal{C}_\Omega(F)\coloneqq\int_{(\G\times \G)\setminus(\Omega^c\times \Omega^c)}\mathrm{sign}(\phi_E(x)-\phi_E(y))(\chi_F(x)-\chi_F(y))K(y^{-1}x)\,dx\,dy.
	\end{equation}
\end{definition}

\begin{remark}\label{rmk:duecinque}
	Let $\Omega\subseteq \G$ is a bounded open set and let $E\subseteq \G$ be measurable. Assume that $\Omega$ is foliated by sub- and super- solutions adapted to $E$. Then, as a consequence of (i) of Definition \ref{def:foliation}, it  immediately follows that  $P_K(E;\Omega)=\mathcal{C}_\Omega(E)$ and
	\begin{equation}
	P_K(F;\Omega) \ge \mathcal{C}_{\Omega}(F)
	\end{equation}
	for every measurable set $F\subseteq \G$ with $F\setminus \Omega= E\setminus \Omega$.
\end{remark}
\begin{proposition}
	\label{prop:propcabre}
	Let $\Omega\subset\G$ be a bounded open set with $P_K(\Omega;\G)<\infty$ and let $E\subset\G$ be a measurable set. Moreover, assume there exists a measurable function $\phi_E\colon\G\rightarrow\R$ satisfying (i) and (ii) of Definition \ref{def:foliation}.
	Then, for every measurable set $F$ such that $F\setminus\Omega=E\setminus\Omega$ and $P_K(F;\G)<\infty$, we have
	\begin{equation}
	\label{eq:calibrationandmeancurv}
	\mathcal{C}_\Omega(F)=2\int_{F\cap\Omega} H_K(\phi_E)(x)\,dx+2\int_{E\setminus\Omega}\int_{\Omega}\mathrm{sign}(\phi_E(x)-\phi_E(y))K(y^{-1}x)\, dx\,dy.
	\end{equation}
\end{proposition}
\begin{proof}
	We introduce the auxiliary kernel $\widetilde K_\varepsilon\colon \G\to [0,+\infty)$ by setting
	\[
	\widetilde{K}_\varepsilon(p)\coloneqq\chi_{\G\setminus B(0,\varepsilon)}(p)K(p),\quad \forall p\in\G.
	\]
	Recalling \eqref{eq:calibrationcabre}, we have that 
	\begin{equation}
	\label{eq:calibracabre}
	\mathcal{C}_\Omega(F)=\lim_{\varepsilon\to 0}\int_{(\G\times \G)\setminus(\Omega^c\times\Omega^c)}\mathrm{sign}(\phi_E(x)-\phi_E(y))(\chi_F(x)-\chi_F(y))\widetilde{K}_\varepsilon(y^{-1}x)\,dx\,dy.
	\end{equation}
	Since $\widetilde K_\varepsilon$ is symmetric we can write
	\begin{equation}
	\label{eq:integralcabre}
	\begin{split}
	&\int_{(\G\times \G)\setminus(\Omega^c\times\Omega^c)}\mathrm{sign}(\phi_E(x)-\phi_E(y))(\chi_F(x)-\chi_F(y))\widetilde{K}_\varepsilon(y^{-1}x)\,dx\,dy= \\
	&2\int_{(\G\times \G)\setminus(\Omega^c\times\Omega^c)}\mathrm{sign}(\phi_E(x)-\phi_E(y))\chi_F(x)\widetilde{K}_\varepsilon(y^{-1}x)\,dx\,dy.
	\end{split}
	\end{equation}
	Now, we split the second integral in \eqref{eq:integralcabre} in two parts: when $y\in F\cap\Omega$ (which implies that $x$ runs through all of $\G$), and when $y\in F\setminus\Omega=E\setminus\Omega$ (which implies that $x$ runs through $\Omega$), and this gives us
	\begin{equation*}
	\begin{aligned}
	&2\int_{(\G\times \G)\setminus(\Omega^c\times\Omega^c)}\mathrm{sign}(\phi_E(x)-\phi_E(y))\chi_F(x)\widetilde{K}_{\varepsilon}(y^{-1}x)\,dx\,dy= \\ 
	&2\int_{F\cap\Omega} \int_{\G}\mathrm{sign}(\phi_E(x)-\phi_E(y))\widetilde{K}_\varepsilon(y^{-1}x) \,dx\,dy
	\\
	&+2\int_{E\setminus\Omega} \int_{\Omega}\mathrm{sign}(\phi_E(x)-\phi_E(y))\widetilde{K}_\varepsilon(y^{-1}x)\,dx\,dy \\
	=&2\int_{F\cap\Omega} H_{\widetilde{K}_{\varepsilon}}[\phi_E](x)dx+2\int_{E\setminus\Omega}\int_{\Omega}\mathrm{sign}(\phi_E(x)-\phi_E(y))\widetilde{K}_\varepsilon(y^{-1}x)\,dx\,dy.
	\end{aligned}
	\end{equation*}
	We notice that, using the notation of \eqref{eq:interact}, one has
	\[
	\int_{E\setminus\Omega}\int_{\Omega} K(y^{-1}x)\,dx\,dy=L_K(E\setminus\Omega,\Omega)\leq L_K(\Omega^c,\Omega)=P_K(\Omega;\G)<+\infty,
	\]
	and, moreover, $|\mathrm{sign}(\phi_E(x)-\phi_E(y))\tilde{K}_\varepsilon(y^{-1}x)|\le K(y^{-1}x)$ for any couple $(x,y)\in E\setminus\Omega\times\Omega$. On the other hand, we know that $H_{\widetilde{K}_\varepsilon}[\phi_E]$ converges in $L^1(\Omega)$ to $H_K[\phi_E]$ as $\varepsilon\to 0$. Therefore, letting $\varepsilon\to 0$ and recalling \eqref{eq:calibracabre}, we conclude the proof.
\end{proof}
\begin{theorem}\label{th:cabre}
	 Let $\Omega\subset\G$ be an open set satisfying $P_{K}(\Omega;\G)<+\infty$ and consider a measurable set $E\subset \G$. Assume that $\Omega$ is foliated by super- and sub- solutions adapted to $E$ and let $\phi_E\colon\G\to \R$ be a measurable function satisfying the assumptions of Definition \ref{def:foliation}. Then the following facts hold.
	\begin{itemize}
		\item[(a)] For every  measurable set $F\subseteq \G$ with $F\setminus \Omega= E\setminus \Omega$ one has
		\[
		P_K(E;\Omega)\leq P_K(F;\Omega).
		\]
		\item[(b)] If $K>0$ and $\phi_E$ is continuous and such that $\mathscr H^Q(\{\phi_E=0\}\cap\Omega)=0$ and if there exists $R>0$ such that $\Omega \subseteq B(0,R)$,  $E\setminus\overline{B(0,R)}\neq \emptyset$ and $(\overline{E})^c\setminus\overline{B(0,R)}\neq\emptyset$, then $E$ is the unique measurable set satisfying $\mathrm{(a)}$ (up to sets of measure zero). 
		\item[(c)] If $H_K[\phi_E](x)=0$ for almost every $x\in \Omega$, then 
		\[
		\mathcal{C}_\Omega(F)=\mathcal{C}_\Omega(E)\,
		\]
		for every measurable set $F\subseteq \G$ with $F\setminus \Omega=E\setminus \Omega$.
	\end{itemize}
\end{theorem}
	\begin{proof}
		(a) By Remark \ref{rmk:duecinque}, it suffices to show that, for every measurable set $F\subseteq \G$ such that $F\setminus\Omega=E\setminus\Omega$, one has
		\[
		\mathcal{C}_{\Omega}(F)\ge \mathcal{C}_{\Omega}(E).
		\]
		We can also assume without loss of generality that $P_K(F;\Omega)<+\infty$. 
		Now, using the fact that $(F\cap\Omega)\cup (E\setminus F)=(E\cap\Omega)\cup (F\setminus E)$ and that both unions are disjoint, we can express the first integral in \eqref{eq:calibrationandmeancurv} as
		\begin{equation}
		\label{eq:dueundici}
		\int_{F\cap\Omega}\!\!\! H_K[\phi_E](x)\,dx=
		\int_{E\cap\Omega}\!\!\! H_K[\phi_E](x)\,dx+
		\int_{F\setminus E}\!\!\! H_K[\phi_E](x)\,dx-
		\int_{E\setminus F}\!\!\! H_K[\phi_E](x)\,dx.
		\end{equation}
		Since both $F\setminus E$ and $E\setminus F$ are contained in $\Omega$, using (iii) of Definition \ref{def:foliation} we get
		\[
		\int_{F\cap \Omega}H_k[\phi_E](x)\,dx\geq \int_{E\cap \Omega}H_K[\phi_E](x)\,dx.
		\]
		Adding to both sides
		\[
		\int_{E\setminus \Omega}\int_\Omega\mathrm{sign}(\phi_E(x)-\phi_E(y))(\chi_F(x)-\chi_F(y))\widetilde K_\varepsilon(y^{-1}x)\,dx\,dy,
		\]
		and recalling that $E\setminus \Omega=F\setminus\Omega$, we conclude the proof of (a).
		
		(b) Assume $\widetilde E\subseteq \G$ is a measurable set such that
		\[
		P_K(\widetilde E;\Omega)\leq P_K(F;\Omega),
		\]
		for every measurable $F\subseteq \G$ with $F\setminus\Omega=E\setminus\Omega$. Then, we have $P_K(E;\Omega)=P_K(\widetilde E;\Omega)$. By \eqref{eq:duetre}, \eqref{eq:calibracabre} and Remark \ref{rmk:duecinque} we have 
		\[
		\begin{aligned}
		\frac 12\int_{(\G\times \G)\setminus(\Omega^c\times \Omega^c)}&|\chi_E(y)-\chi_E(x)|K(y^{-1}x)\,dy\,dx\\&=\int_{(\G\times \G)\setminus(\Omega^c\times \Omega^c)}\mathrm{sign}(\phi_E(y)-\phi_E(x))(\chi_{\widetilde{E}}(y)-\chi_{\widetilde{E}}(x))K(y^{-1}x)\,dy\,dx.
		\end{aligned}
		\]
		Since $K>0$, we get 
		\begin{equation}
		\label{eq:duedodici}
		\phi_E(x)>\phi_E(y)\quad\text{for a.e.\ } (x,y)\in((\widetilde E\cap\Omega)\times \widetilde E^c)\cup(\widetilde E\times(\widetilde E^c\cap\Omega)).
		\end{equation}
		By hypothesis, the function $\phi_E$ takes both a positive and a negative value in $\overline{B(0,R)}^c$. Since $\phi_E$ is continuous, for every $\delta>0$ small enough, both $\{-\delta<\phi_E<0\}\setminus\overline{B(0,R)}$ and $\{0<\phi_E<\delta\}\setminus\overline{B(0,R)}$ are nonempty open sets. Hence, since $\widetilde E^c\setminus\overline{B(0,R)}=\{\phi_E\leq 0\}\setminus\overline{B(0,R)}$ and $\widetilde E\setminus\overline{B(0,R)}=\{\phi_E>0\}\setminus\overline{B(0,R)}$, from \eqref{eq:duedodici}, by letting $\delta\to 0$, we deduce 
		\[
		\phi_E(x)\ge 0\quad\text{for a.e.\ $x\in \widetilde E\cap\Omega$}\quad\text{and}\quad\phi_E(y)\le 0\quad\text{for a.e.\ }y\in \widetilde E^c\cap\Omega.
		\]
		Since the set $\{\phi_E=0\}\cap\Omega$ has zero measure by assumption, we deduce that $\widetilde E\cap\Omega=\{\phi_E>0\}\cap\Omega=E\cap\Omega$ up to a measure zero set.
		
		(c) If $H_K[\phi_E](x)=0$ for almost every $x\in \Omega$, the last two integrals in \eqref{eq:dueundici} vanish, and thus \eqref{eq:calibrationandmeancurv} leads to $\mathcal{C}_\Omega(F)=\mathcal{C}_\Omega(E)$, for every measurable set $F\subseteq \G$ with $F\setminus \Omega=E\setminus\Omega$.
	\end{proof}

\begin{remark}
	\label{rem:similarities}
	Notice that Theorem \ref{th:calibrationimpliesminimizer} and Theorem \ref{th:cabre} present many similarities. On the one hand, Theorem \ref{th:calibrationimpliesminimizer} holds for general measurable functions and, moreover, if a set $\Omega$ admits sub- and super- solutions adapted to $E$, then $\zeta(x,y)=\mathrm{sign}(\phi_E(y)-\phi_E(x))$ is a calibration for $\chi_E$. However, although Theorem \ref{th:cabre} requires more assumptions, it gives us additional information about the local minimality of the functional $\mathcal C_\Omega$ and about the uniqueness of minimizers, which was only known for the specific case of halfspaces, as shown in Theorem \ref{th:localminimalitycalibration}.
\end{remark}

\section{$\Gamma$-convergence of the rescaled functionals}

In this section we analyze the $\Gamma$-limit of the rescaled sequence $\frac{1}{\varepsilon}P_{K_\varepsilon}(E_\varepsilon;\Omega)$, where $(E_\varepsilon)_{\varepsilon>0}$ is a family of measurable sets converging in $L^1(\Omega)$ to some set $E\subseteq\Omega$. In the study of the asymptotic behavior of the functionals, one has to deal with in the blow-up of sets of finite perimeter. In the setting of Carnot groups, one of the main and still unsolved problem concerns the regularity of the (reduced) boundary of a set of finite perimeter. The solution of this problem in the Euclidean spaces goes back to De Giorgi \cite{Deg55}. He proved that the reduced boundary of a set of finite perimeter in ${\mathbb R}^n$ is $(n-1)$-rectifiable, i.e.,\ it can be covered, up to a set of $\mathscr H^{n-1}$-measure zero, by a countable family of $C^1$-hypersurfaces. The validity of such a result has deep consequences in the development of Geometric Measure Theory and Calculus of Variations (see e.g.\ the monographs \cite{AFP, EG}).

The validity of a rectifiability-type Theorem in the context of Carnot groups is still not yet known in full generality. However, there are complete results in all Carnot groups of step 2 (see \cite{FSSC01, FSSC03}) and in the so-called Carnot groups of type $\star$, see \cite{Marchi}, which generalize the class of step 2. 
In these papers the authors show that the reduced boundary of a set of finite perimeter in a Carnot group of the chosen class is rectifiable with respect to the intrinsic structure of the group. 

Motivated by these results, we introduce the following notation, see \cite{DonVittone} that will be used in Theorem \ref{th:gammaliminf} and Remark \ref{rem:rettificabilita}. Also recall Definitions \ref{def:rectifiable} and \ref{def:reducedboundary}.
\begin{definition}
	\label{def:proprietaR}
	We say that a Carnot group $\G$ satisfies property $\mathcal R$ if every set $E\subseteq \G$ of locally finite perimeter in $\G$ has rectifiable reduced boundary.
\end{definition}
As already mentioned before, property $\mathcal R$ is satisfied in Euclidean spaces, in all Carnot groups of step $2$ and in the so-called Carnot groups of type $\star$.

The first part of this section is devoted to the proof of a compactness criterion for the rescaled family $\frac 1\varepsilon P_{K_\varepsilon}$, see Theorem \ref{th:compactness}. The final part of this section deals with the estimate of the $\Gamma$-liminf for the same rescaled family of functionals, in the class of Carnot groups satisfying property $\mathcal R $.
We start with the following

\begin{proposition}\label{prop:piove}
	Let $E,F\subseteq \G$ be measurable sets. Then the following fact hold:
		 If $N\subseteq \G$ is a set of finite perimeter in $\G$ such that $E\subseteq N$ and $F\subseteq N^c$, then
		\[
		\limsup_{\varepsilon\to0}\frac 1\varepsilon L_\varepsilon (E,F)\leq \frac {P_\G(N)}2  \int_\G K(\xi)d(\xi,0)\,d\xi.
		\] 
\end{proposition}
	\begin{proof}
		By a change of variables and Proposition \ref{prop:frechetkolmogorov} we have
		\[
		\begin{aligned}
		\frac 1\varepsilon L_\varepsilon (E,F)&\leq \frac 1\varepsilon \int_N\int_{N^c}\frac 1{\varepsilon^Q}K(\delta_{1/\varepsilon}(y^{-1}x))\,dydx\\
		&=\frac 1{2\varepsilon}\int_\G\int_\G K(g)|\chi_N(x\delta_\varepsilon g)-\chi_N(x)|\,dgdx\\
		&\leq \frac{P_\G(N)}{2}\int_\G K(\xi)d(\xi,0)\,d\xi.\qedhere
		\end{aligned}
		\]
	\end{proof}

	 Before the proof of the compactness Theorem, we remark the validity of the following fact, whose proof is an immediate calculation. We denote by $J_G$ the functional in \eqref{eq:J1J2} with kernel $G$ and by $P_G$ the corresponding perimeter.
	\begin{lemma}\label{lemma:disuguaglianzaconvoluta}
	Let $G\in L^1(\mathbb{G})$ be a positive function. Then, for any $u\in L^{\infty}(\G)$ it holds that
	\[
	\int_{\G}\int_{\mathbb{G}}(G\ast G)(y)|u(xy)-u(x)|\;dydx\leq 2\left\|G\right\|_{L^1(\mathbb{G})}J_G(u;\mathbb{G}).
	\]
	In particular, if we choose $u=\chi_E$ we have
	\[
	\int_{\G}\int_{\mathbb{G}}(G\ast G)(y)|\chi_E(xy)-\chi_E(x)|\;dydx\leq 4\left\|G\right\|_{L^1(\mathbb{G})}P_G(E).
	\]
	\end{lemma}
\noindent We are ready to prove the compactness result.
\begin{theorem}\label{th:compactness}
	Let $\Omega\subseteq \G$ be a bounded open set. Let $(\varepsilon_n)$ be an infinitesimal sequence of positive numbers and let $(E_n)$ be a sequence of measurable sets in $\Omega$. Assume that there exists $C>0$ such that
	\begin{equation}\label{eq:equi}
	\frac 1{\varepsilon_n}P_{\varepsilon_n}(E_n;\Omega)\leq C,\quad \forall n\in \mathbb N.
	\end{equation}
	Then, there exist a subsequence $(E_{n_k})$ of $(E_n)$ and a set $E$ of finite perimeter in $\Omega$ such that $(E_{n_k})$ converges to $E$ in $L^1(\Omega)$.
\end{theorem}
\begin{proof}
	We write $E_\varepsilon$ in place of $E_n$, to avoid inconvenient notation. Fix a ball $B$ in $\G$ such that $\Omega\subseteq B$. For any positive $\varphi\in C_c^\infty(\G)\setminus \{0\}$ we define, for every $\varepsilon\in(0,1)$, the map
	\[
	\varphi_\varepsilon(x)\coloneqq \frac 1{ \varepsilon^Q \int_\G \varphi(\xi)\;d\xi}\varphi(\delta_{1/\varepsilon}x),
	\]
	and we consequently set $v_\varepsilon\coloneqq  \chi_{E_\varepsilon}*\varphi_\varepsilon$. We can therefore estimate
	\begin{equation}\label{eq:viepsilon}
	\begin{aligned}
	\int_\G |v_\varepsilon(\xi)-\chi_{E_\varepsilon}(\xi)|\;d\xi & \leq \int_\G \int_\G\varphi_\varepsilon(\eta^{-1}\xi)|\chi_{E_\varepsilon}(\eta)-\chi_{E_\varepsilon}(\xi)|\; d\eta d\xi\\&
	=\int_\G\int_\G \varphi_\varepsilon(\xi)|\chi_{E_\varepsilon}(\eta)-\chi_{E_\varepsilon}(\eta\xi)|\; d\eta d\xi.
	\end{aligned}
	\end{equation}
	Notice that, by definition of $\varphi_\varepsilon$ and since $\varphi$ has compact support, the families $(v_\varepsilon)$ and $(\chi_{E_\varepsilon})$ share the same limits in $L^1(\G)$. Reasoning in a similar way on the horizontal gradient of $v_\varepsilon$ we get
	\begin{equation}\label{eq:gradient}
	\begin{aligned}
	\int_\G|\nabla_X v_\varepsilon(\xi)|\;d\xi&=\int_\G \left|\int_\G \nabla_X \varphi_\varepsilon(\eta^{-1}\xi)\chi_{E_\varepsilon}(\eta)\,d\eta\right|d\xi\\&
	\leq \int_\G \int_\G |\nabla_X \varphi_\varepsilon(\eta^{-1}\xi)||\chi_{E_\varepsilon}(\eta)-\chi_{E_\varepsilon}(\xi)|\,d\eta d\xi
	\\&
	\hphantom{=}+\int_\G \chi_{E_\varepsilon}(\xi)\left|\int_\G \nabla_X \varphi_\varepsilon(\eta^{-1}\xi)\,d\eta\right| d\xi\\&
	=\int_\G\int_\G |\nabla_X\varphi_\varepsilon(\xi)||\chi_{E_\varepsilon}(\eta\xi)-\chi_{E_\varepsilon}(\eta)|\;d\eta d\xi.
	\end{aligned}
	\end{equation}
	Notice that the identity
	\[
	\int_\G\nabla_X \varphi_\varepsilon(\eta^{-1}\xi)\,d\eta=0,
	\]
	holds since $\varphi\in C_c^\infty(\G)$ and horizontal vector fields in Carnot groups are divergence-free (see e.g.\ \cite[Proposition 1.3.8.]{BonLanUgu}).
	Define now the map 
	\[
	T(s)\coloneqq
	\begin{cases}
	s& \text{ if $|s|\leq 1$,}\\
	1 & \text{ otherwise},
	\end{cases}
	\]
	and consider the truncated kernel $G\coloneqq T\circ K$. We notice that $G\geq 0$ and $G\in L^\infty(\G)$. Moreover, by \eqref{eq:stimaL1} and the fact that $T(s)\leq s$ for any $s\in [0,\infty)$, we can estimate
	\[
	\int_\G |(T\circ K)(\xi)| \,d\xi\leq \int_{B(0,1)}\,d\xi+\int_{\G\setminus B(0,1)}K(\xi)\,d\xi<\infty,
	\]
	which implies that $G\in L^1(\G)$. Since $G\in L^1(\G)\cap L^\infty(\G)$, the map $G*G$ is continuous. This is a consequence of the following estimate
	\begin{align*}
	|(G*G)(p)-(G*G)(q)|&\leq \int_{\G} G(\xi)|G(p\xi^{-1})-G(q\xi^{-1})|\, d\xi\\
	&\leq \|G\|_{L^\infty(\G)}\int_{\G}|G(p\xi^{-1})-G(q\xi^{-1})|\, d\xi \\
	&=\|G\|_{L^\infty(\G)}\|\tau_{q^{-1}p}G-G\|_{L^1(\G)},
	\end{align*}
	and Corollary \ref{mix}. We now choose a positive $\varphi\in C_c^\infty(\G)\setminus \{0\}$ such that 
	\[
	\varphi\leq G*G \quad \text{and}\quad|\nabla_X \varphi|\leq G*G.
	\]
	We can assume without loss of generality that $v_\varepsilon \in C_c^\infty(B)$ for every $\varepsilon \in (0,1)$. Setting $G_\varepsilon(\xi)\coloneqq\varepsilon^{-Q}G(\delta_{1/\varepsilon}\xi)$,  and taking \eqref{eq:viepsilon} and \eqref{eq:gradient} into account we obtain
	\begin{equation}\label{eq:tresette}
	\int_\G|v_\varepsilon(\xi)-\chi_{E_\varepsilon}(\xi)|\;d\xi\leq\int_\G\int_\G(G_\varepsilon*G_\varepsilon)(\xi)|\chi_{E_\varepsilon}(\eta\xi)-\chi_{E_\varepsilon}(\eta)|\;d\eta d\xi,
	\end{equation}
	and 
	\begin{equation}\label{eq:tresettebis}
	\int_\G|\nabla_X v_\varepsilon(\xi)|\;d\xi\leq \frac 1\varepsilon\int_\G\int_\G(G_\varepsilon*G_\varepsilon)(\xi)|\chi_{E_\varepsilon}(\eta\xi)-\chi_{E_\varepsilon}(\eta)|\;d\eta d\xi,
	\end{equation}
	where the last inequality comes from the fact that 
	\[
	(\nabla_X\varphi_\varepsilon)(\xi)=\frac 1{\varepsilon^{Q+1}}(\nabla_X\varphi)(\delta_{1/\varepsilon}\xi),
	\] 
	and
	\[
	(G_\varepsilon*G_\varepsilon)(\xi)=\frac 1{\varepsilon^Q}(G*G)(\delta_{1/\varepsilon} \xi).
	\]
	By applying Lemma \ref{lemma:disuguaglianzaconvoluta} and since $E_\varepsilon\subseteq \Omega$ for each 
	$\varepsilon>0$, we have
	\[
	\begin{aligned}
	\int_\G\int_\G (G_\varepsilon *G_\varepsilon)(\xi) &|\chi_{E_\varepsilon}(\eta\xi)-\chi_{E_\varepsilon}(\eta)|\;d\eta d\xi \leq 4 \|G\|_{L^1(\G)} P_{G_\varepsilon}(E_\varepsilon)
	\\&\leq4\|G\|_{L^1(\G)}P_{K_\varepsilon}(E_\varepsilon)= 4\|G\|_{L^1(\G)}\left ( \frac 12 J_\varepsilon ^1(E_\varepsilon;\Omega)+J^2_\varepsilon(E_\varepsilon;\Omega)\right )
	\\&=4\|G\|_{L^1(\G)} J_\varepsilon(E_\varepsilon,\Omega).
	\end{aligned}
	\]
	Condition \eqref{eq:equi} then gives $M>0$ such that 
	\[
	\frac 1\varepsilon\int_\G\int_\G (G_\varepsilon *G_\varepsilon)(\xi) |\chi_{E_\varepsilon}(\eta\xi)-\chi_{E_\varepsilon}(\eta)|\;d\eta d\xi\leq M\|G\|_{L^1(\Omega)}.
	\]
	By the estimates \eqref{eq:tresette} and \eqref{eq:tresettebis} we get that $(v_\varepsilon)$ is equibounded in $W^{1,1}_\G(B)$. Then, by the general version of Rellich-Kondrakov's Compactness Theorem in metric measure spaces, (see \cite[Theorem 8.1]{HajKos}), up to subsequences, $v_\varepsilon$ converges in $L^1(B)$ to some $w$. We moreover observe that \eqref{eq:tresette} also tells us that $w=\chi_{\widetilde E}$ for some $\widetilde E$ with finite measure in $B$. Inequality \eqref{eq:tresettebis} together with the lower semicontinuity of the total variation implies that $\widetilde E$ has finite perimeter in $B$. By setting $E\coloneqq\widetilde E\cap \Omega $, we have that $E$ has finite perimeter in $\Omega$ and, by \eqref{eq:viepsilon}, ${E_\varepsilon}\to E$ in $L^1(\Omega)$.
\end{proof}

\begin{rmk}
	 In case $\Omega$ has finite perimeter and the stronger integrability condition
	\begin{equation}\label{eq:integrabilitaforte}
	\int_\G K(x)d(x,0)\, dx<+\infty
	\end{equation}
	is satisfied, then Theorem \ref{th:compactness} can be strengthened replacing condition \eqref{eq:equi} with the weaker
	\[
	\frac 1{\varepsilon_n}J^1_{\varepsilon_n}(E_n,\Omega)\leq C,\quad \forall n\in \mathbb N.
	\]
	Indeed, applying (i) of Proposition \ref{prop:piove} with $N=\Omega$ one gets some $C_2>0$ such that
	\begin{align*}
	\frac 1{\varepsilon_n}J^2_{\varepsilon_n}(E_{\varepsilon_n},\Omega)&=
	\frac 1{\varepsilon_n}L_{\varepsilon_n}(\Omega\cap E_{\varepsilon_n}, \Omega^c \cap E_{\varepsilon_n}^c)
	\\
	&\leq \frac 12 P_\G(\Omega)\int_\G K(x)d(x,0)\, dx \leq C_2,  \quad \forall n\in \mathbb N.
	\end{align*}
	Notice however that condition \eqref{eq:integrabilitaforte} is in contrast with \eqref{eq:infcappa} below, that will be used in Theorem \ref{th:gammaliminf}.
\end{rmk}

Denote for shortness $B\coloneqq B(0,1)$. For every halfspace $H\subseteq \G$ we set
\begin{equation}\label{eq:biacca}
b(H)\coloneqq\inf\left \{\liminf_{\varepsilon\to 0}\frac 1{2\varepsilon}J_\varepsilon^1(E_\varepsilon, B(0,1)): E_\varepsilon\to H \text{ in $L^1(B(0,1))$}\right \}. 
\end{equation}
A priori, the quantity $b(H)$ defined above might depend on the halfspace $H$. In the following proposition, we find sufficient conditions on the kernel in order to have a uniform positive lower bound on $b$. In Remark \ref{rem:libero}, we observe that, in free Carnot groups, the function $b$ defined above is constant.
\begin{proposition}\label{prop:biacca}
		Assume there exists a monotone decreasing $\widetilde K\colon [0,+\infty)\to [0,+\infty)$ such that $K(\xi)=\widetilde K(\|\xi\|)$ for every $\xi\in \G$ and that
		\begin{equation}\label{eq:infcappa}
		\inf_{r>1}\widetilde K(r)r^{Q+1}>0.
		\end{equation}
		Then 
		\[
		\inf\{b(H): H \text{ is a vertical halfspace}\}>0.
		\]
	
\end{proposition}
\begin{proof}
	 Fix a halfspace $H$. We first prove that $b(H)>0$. By definition of $b(H)$ and a diagonal argument, there exists a family $\chi_{E_{\varepsilon}}$ that converges to $\chi_H$ in $L^1(B)$ as $\varepsilon\to 0$ such that
	\[
	\liminf_{\varepsilon\to 0}\frac{1}{2\varepsilon}J^1_\varepsilon(E_\varepsilon;B)=b(H).
	\]
	Thanks to Severini-Egorov's Theorem there exists an open set $A\subseteq B$ such that 
	\begin{equation}\label{eq:Egorov}
	\mathscr H^Q(B\setminus A)<\frac{\mathscr H^Q(H\cap B)}{2}
	\end{equation}
	 and $\chi_{E_\varepsilon}$ converges to $\chi_H$ uniformly on $A$, as $\varepsilon\to 0$. We therefore find $\varepsilon_0$ such that 
	\[
	\sup_{x\in A}|\chi_{E_\varepsilon}(x)-\chi_H(x)|<1,\quad \forall\varepsilon\leq\varepsilon_0,
	\]
	and hence, for every $\varepsilon\leq \varepsilon_0$ we have $E_\varepsilon\cap A=H\cap A\eqqcolon C^+$. By reasoning in the same way on $E_\varepsilon^c$, we may assume without loss of generality that, for every $\varepsilon\leq \varepsilon_0$, we also have $E_\varepsilon^c\cap A=H^c\cap A\eqqcolon C^-$.
	Notice that, by \eqref{eq:Egorov}, we have 
	\begin{equation}\label{eq:posmin}
	\min\{\mathscr H^Q(C^+),\mathscr H^Q( C^-)\}>0.
	\end{equation}
	For every $\varepsilon\leq \varepsilon_0$, we have
	\[
	\begin{aligned}
	\frac 1{2\varepsilon}J_\varepsilon^1(E_\varepsilon;B) &= \frac 1{\varepsilon} \int_{E_\varepsilon}\int_{E_\varepsilon^c\cap B}K_\varepsilon(y^{-1}x)\,dydx\geq \varepsilon^{Q-1} \int_{\delta{1/\varepsilon}C^+}\int_{\delta{1/\varepsilon}C^-}K(y^{-1}x)\,dydx \\
	&\geq \varepsilon^{Q-1}\widetilde K(\mathrm{diam}(\delta_{1/\varepsilon}C^+\cup \delta_{1/\varepsilon}C^-))\mathscr H^Q(\delta_{1/\varepsilon}C^+)\mathscr H^Q(\delta_{1/\varepsilon}C^-)\\
	&= \frac 1{\varepsilon^{Q+1}}\widetilde K\left (\frac{\mathrm{diam}(C^+\cup C^-)}{\varepsilon}\right )\mathscr H^Q(C^+)\mathscr H^Q(C^-),
	\end{aligned}
	\]
	which, by \eqref{eq:infcappa} and \eqref{eq:posmin}, is a positive lower bound independent of $\varepsilon$.
	
	To conclude the proof of (i), it is enough to check that $b$ is lower-semicontinuous. In fact, if this were true, by the compactness of the sphere $\mathbb S^{m-1}$, we would have that $b$ admits a minimum, that, by the previous step would be strictly positive.\\
Let $\nu_{\eta}\in \mathbb{S}^{m-1}$ such that $\nu_{\eta}\to \nu$ as $\eta\to 0$ and let $H_{\eta}$ be the family of vertical halfspace associated to $\nu_{\eta}$. Then $\chi_{H_{\nu_\eta}}\to \chi_{H_{\nu}}$ in $L^1(B)$ as $\eta\to 0$.

	Fix $\sigma>0$. For every $\eta>0$ we can find $F_\varepsilon^\eta$ converging to $H_\eta$ in $L^1(B)$, as $\varepsilon\to 0$ such that
	\[
	\liminf_{\varepsilon\to 0}\frac 1{2\varepsilon}J^1_\varepsilon (F_\varepsilon^\eta;B)\leq b(H_\eta)+\sigma.
	\]
	Considering $E_\varepsilon\coloneqq F_\varepsilon^\varepsilon$, we easily find that $E_\varepsilon\to H$ in $L^1(B)$, as $\varepsilon\to 0$ and hence
	\[
	b(H)\leq \liminf_{\varepsilon\to 0}\frac 1{2\varepsilon}J^1_\varepsilon (E_\varepsilon;B)\leq \liminf_{\varepsilon \to0} b(H_\varepsilon)+\sigma.
	\]
	The thesis follows by the arbitrariness of $\sigma$.
\end{proof}
	
	\begin{remark} \label{rem:libero}
If $\G$ is a free Carnot group (we refer to \cite[p.\ 45]{VSCC} or \cite[p.\ 174]{Varadarajan} for the definition) and $K$ is radial, then, if $H_1,H_2\subseteq \G$ are vertical halfspaces in $\G$, one has $b(H_1)=b(H_2)$. Indeed,	let $\nu_1,\nu_2\in\g_1 \setminus \{0\}$ such that $H_1=H_{\nu_1}$ and $H_2=H_{\nu_2}$. It is enough to show that $b(H_1)\leq b(H_2)$. Let $E_\varepsilon^2$ be a family of measurable sets in $B$ such that $E_\varepsilon^2\to H_{\nu_2}$ in $L^1(B)$ as $\varepsilon\to 0$. Now consider an orthogonal isomorphism $T\colon\g_1\to\g_1$ such that $T(\nu_2)=\nu_1$. Since $\G$ is free, the map $T$ extends in a unique way to a Lie algebra isomorphism $T\colon\g\rightarrow\g$ that induces an isometry $I\colon\G\rightarrow \G$ defined by
	\[
	I\coloneqq\exp\circ T\circ\log.
	\] 
	We claim that $I(H_2)=H_1$. Indeed, for every $\xi\in \G$, one has 
	\[
	\begin{aligned}
	 \langle\pi_1\log \xi, \nu_1 \rangle =&\langle\pi_1\log \xi, T(\nu_2) \rangle
	=\langle T(\pi_1\log \xi), \nu_2 \rangle\\=&\langle\pi_1 T(\log \xi), \nu_2 \rangle=\langle\pi_1 \log I(\xi), \nu_2 \rangle.
	\end{aligned}
	\]
	Since $K$ is radial and $I$ is an isometry, it is easy to see that $J^1(A;B)=J^1(I(A);I(B))$. By noticing that $I(B)=B$ and that $I(E_\varepsilon^2)\to H_1$ in $L^1(B)$ as $\varepsilon\to 0$, we have that
	\[
	b(H_1)\leq \liminf_{\varepsilon\to 0}\frac 1{2\varepsilon}J^1_\varepsilon(I(E_\varepsilon^2);B)=\liminf_{\varepsilon\to0}\frac 1{2\varepsilon}J^1_\varepsilon(E_\varepsilon^2;B),
	\]
	whence $b(H_1)\leq b(H_2)$.
\end{remark}

	\begin{remark}\label{rem:rettificabilita}
	Let $\G$ be a Carnot group satisfying property $\mathcal R$ and let $E$ be a set of locally finite perimeter in some open set $\Omega\subseteq \G$. Then, by \cite[Lemma 3.8]{FSSC03}, if $\G$ satisfies property $\mathcal R$, for every $p\in \mathcal FE$ one has 
	\begin{equation}\label{eq:densityper}
	\lim_{r\to0}\frac{P_\G(E;B(p,r))}{r^{Q-1}}=P_\G(H_{\nu_E(p)}; B(0,1))\eqqcolon\vartheta(\nu_E(p)).
	\end{equation}
	Notice also that, since $H_\nu$ has smooth boundary for any $\nu \in \g$, its perimeter can be explicitly computed (up to identification of $\G$ with $\R^n$ by means of exponential coordinates) to get 
	\begin{equation}
	\label{eq:varteta}
	\vartheta(\nu) =\mathscr H_e^{n-1}(\partial H_\nu \cap B(0,1)),
	\end{equation}
	where $\mathscr H_e^{n-1}$ denotes the $(n-1)$-dimensional Hausdorff measure with respect to the Euclidean metric (see e.g.\ \cite[Theorem 5.1.3]{Monti} and \cite[Proposition 2.22]{FSSC03}).
	\end{remark}

\begin{theorem}\label{th:gammaliminf}
	Let $\G$ be a Carnot group satisfying property $\mathcal R$, let $\Omega\subseteq \G$ be open and bounded and let $K\colon \G\to [0,+\infty)$ be a radial decreasing kernel satisfying \eqref{eq:nonnegativekernel}, \eqref{eq:evenkernel} \eqref{eq:stimaL1} and \eqref{eq:infcappa}.
	Then, there exists $\rho\colon \g_1\to (0,+\infty)$ such that, for every family $(E_\varepsilon)$ of measurable sets converging in $L^1(\Omega)$ to $E\subseteq \Omega$, one has
	\begin{equation}\label{eq:gammaliminf}
	\int_\Omega \rho(\nu_E) \,dP_\G(E;\cdot)\leq \liminf_{\varepsilon\to0} \frac{1}{\varepsilon}P_\varepsilon(E_\varepsilon;\Omega).
	\end{equation}
	More precisely, for every $\nu\in \g_1$, the function $\rho$ can be represented as:
	\[
	\rho(\nu)=\frac{b(H_\nu)}{\vartheta(\nu)},
	\]
	where $b$ and $\vartheta$ are respectively defined as in \eqref{eq:biacca} and \eqref{eq:densityper}.
\end{theorem}
\begin{proof}
	Fix $\varepsilon>0$. We define the function
	\[
	f_\varepsilon(\xi)\coloneqq\begin{cases}
	\displaystyle\frac {1}{2\varepsilon}\int_{E_\varepsilon^c\cap \Omega}K_\varepsilon(\eta^{-1}\xi)\,d\eta +\frac 1\varepsilon\int_{\Omega^c\cap E_\varepsilon^c} K(\eta^{-1}\xi)\,d\eta, & \text{if $\xi\in E_\varepsilon$}\\
\displaystyle	\frac 1{2\varepsilon}\int_{E_\varepsilon\cap \Omega}K_\varepsilon(\eta^{-1}\xi)\,d\eta, & \text{if $\xi\in E_\varepsilon^c$,}
	\end{cases}
	\]
	and set $\mu_\varepsilon\coloneqq f_\varepsilon \mathcal H^Q\res\Omega$. Notice that
	\[
	\|\mu_\varepsilon\|\coloneqq\mu_\varepsilon(\Omega)=\frac 1{\varepsilon}P_\varepsilon(E_\varepsilon;\Omega).
	\]
	Without loss of generality we can assume that there exists $M>0$ such that
	\[
	\frac 1{\varepsilon}P_\varepsilon(E_\varepsilon;\Omega)\leq M, \quad\forall \varepsilon>0.
	\]
	By this uniform bound and the assumptions on $\Omega$, we get that, by Theorem \ref{th:compactness}, $E$ has finite perimeter in $\Omega$. We set for shortness $P_{E}\coloneqq P_{\G}(E;\cdot)$. Moreover, thanks to the weak* compactness of measures, we can find a positive measure $\nu$ such that $\mu_\varepsilon\rightharpoonup^\ast\mu$ as $\varepsilon\to0$ up to subsequences, and hence
	\[
	\|\mu\|\leq \liminf_{\varepsilon\to0}\|\mu_\varepsilon\|.
	\]
	To prove \eqref{eq:gammaliminf}, it is enough to show that
	\[
	\|\mu\|\geq \int_\Omega \rho(\nu_E) \, dP_{E},
	\]
	for some $\rho\colon \g_1\to (0,+\infty)$ that will be determined in the sequel. Notice that, since by \cite{Ambrosio02} the perimeter measure is asymptotically doubling, we are allowed to differentiate $\mu$ with respect to the perimeter $P_{E}$, see \cite[Theorem 2.8.17]{Federer}.  We then aim to prove that
	\[
	\frac{d\mu}{dP_{E}}(p)\geq \rho(\nu_E(p)), \quad \text{for $P_{E}$-a.e. $p\in \Omega$}, 
	\]
	where$\frac{d\mu}{dP_{E}}(p)$ denotes the Radon-Nikodym derivative of $\mu$ with respect to $P_{E}$.
	Fix $p\in \mathcal FE \cap \Omega$. Since $\G$ satisfies property $\mathcal R$, by \eqref{eq:densityper} we have 
	\[
	\frac{d\mu}{dP_{E}}(p)=\lim_{r\to0}\frac{\mu(B(p,r)}{P_E(B(p,r))}=\frac 1{\vartheta(\nu_E(p))}\lim_{r\to 0}\frac{\mu(B(p,r))}{r^{Q-1}}.
	\]
	Since $\mu_\varepsilon$ weakly$^*$ converges to $\mu$ as $\varepsilon\to 0$, we have that $\mu_\varepsilon(B(p,r))$ converges to $\mu(B(p,r))$ for every $r>0$ outside a countable subset $Z\subseteq (0,+\infty)$ of radii. We therefore have
	\[
	\frac{d\mu}{dP_{E}}(p)=\frac 1{\vartheta(\nu_E(p))}\lim_{r\to 0, r\notin Z}\left ( \lim_{\varepsilon \to 0} \frac {\mu_\varepsilon (B(p,r))}{r^{Q-1}}\right ).
	\]
	By a diagonal argument, we may choose two infinitesimal sequences $(\varepsilon_j)$ and $(r_j)$ such that
	\[
	\lim_j\frac{\varepsilon_j}{r_j}=0,
	\]
	and so that
	\[
	\frac{d\mu}{dP_{E}}(p)=\frac 1{\vartheta(\nu_E(p))}\lim_j\frac{\mu_{\varepsilon_j}(B(p,r_j))}{r_j^{Q-1}}.
	\]
	By making the computation explicit, we can write
	\[
	\begin{aligned}
	\frac{d\mu}{dP_{E}}(p)=\frac 1{\vartheta(\nu_E(p))}\lim_j\frac 1{\varepsilon_j r_j^{Q-1}}\Bigg (\frac 12& \int_{E_{\varepsilon_j}\cap\Omega\cap B(p,r_j)}\int_{E_{\varepsilon_j}^c\cap \Omega} K_{\varepsilon_j}(y^{-1}x)\,dydx\\
	+\frac 12&\int_{E_{\varepsilon_j}^c\cap\Omega\cap B(p,r_j)}\int_{E_{\varepsilon_j}\cap \Omega} K_{\varepsilon_j}(y^{-1}x)\,dydx \\
	+&\int_{E_{\varepsilon_j}\cap\Omega\cap B(p,r_j)}\int_{\Omega^c\cap E_\varepsilon^c} K_{\varepsilon_j}(y^{-1}x)\,dy dx \Bigg ),	
	\end{aligned}
	\]
	and hence, since $P_\varepsilon=J_\varepsilon\geq \frac 12J^1_\varepsilon$ and since, for $j$ sufficiently large, one has $B(p,r_j)\subseteq \Omega$, we get 
	\[
	\begin{aligned}
	\frac{d\mu}{dP_{E}}(p)&\geq \frac 1{\vartheta(\nu_E(p))} \liminf_j \frac 1{2\varepsilon_j r_j^{Q-1}}J_{\varepsilon_j}^1(E_{\varepsilon_j}; B(p,r_j)\cap \Omega)\\&=\frac 1{\vartheta(\nu_E(p))}\liminf_j \frac 1{2\varepsilon_j r_j^{Q-1}}J_{\varepsilon_j}^1(E_{\varepsilon_j}; B(p,r_j)).
	\end{aligned}
	\]
	By a change of variable, since $J^1$ is left unchanged by isometries, we have
	\[
	J^1_{\varepsilon_j}(E_{\varepsilon_j}; B(p,r_j))=r_j^QJ^1_{{\varepsilon_j}/{r_j}}\left (\delta_{1/{r_j}}p^{-1}E_{\varepsilon_j}; B\right ).
	\]
	This implies that
	\[
	\frac{d\mu}{dP_{E}}(p)\geq\frac 1{\vartheta(\nu_E(p))}\liminf_j\frac{r_j}{2\varepsilon_j}J^1_{{\varepsilon_j}{r_j}}\left (\delta_{1/{r_j}}p^{-1}E_{\varepsilon_j}; B\right ).
	\]
	Since, by property $\mathcal R$, the sequence $\delta_{1/{\varepsilon_j}}p^{-1}E_{\varepsilon_j}$ converges to $H_{\nu_E(p)}$ in $L^1(B)$ as $j\to\infty$ we get 
	\[
	\frac{d\mu}{dP_{E}}(p)\geq \frac 1{\vartheta(\nu_E(p))}\, b(H_{\nu_E(p)}). \qedhere
	\]
\end{proof}

\bibliographystyle{acm}
\bibliography{mini3}
\include{bibliography}
\end{document}